\documentclass[aap]{imsart}

\RequirePackage{amsthm,amsmath,amsfonts,amssymb}
\RequirePackage[numbers,sort&compress]{natbib}
\RequirePackage[colorlinks,citecolor=blue,urlcolor=blue]{hyperref}
\RequirePackage{graphicx}
\RequirePackage{booktabs}
\RequirePackage{subcaption}
\RequirePackage{multirow}

\startlocaldefs
\numberwithin{equation}{section}
\theoremstyle{plain}

\newtheorem{theorem}{Theorem}[section]
\newtheorem{lemma}[theorem]{Lemma}
\newtheorem{proposition}[theorem]{Proposition}
\newtheorem{remark}[theorem]{Remark}
\newtheorem{corollary}[theorem]{Corollary}
\theoremstyle{definition}

\newtheorem*{fact}{Fact}
\renewcommand\le{\leqslant}
\renewcommand\ge{\geqslant}

\newcommand{\refT}[1]{Theorem~\ref{#1}}
\newcommand{\refTs}[1]{Theorems~\ref{#1}}

\newcommand{\refL}[1]{Lemma~\ref{#1}}

\newcommand{\refS}[1]{Section~\ref{#1}}

\newcommand{\refSS}[1]{Section~\ref{#1}}

\begingroup
  \divide\count255 by 60
  \multiply\count255 by -60
  \advance\count255 by \time
  \ifnum \count255 < 10 \xdef\klockan{\the\count1.0\the\count255}
  \else\xdef\klockan{\the\count1.\the\count255}\fi
\endgroup

\newcommand\nopf{\qed}   

\begingroup
  \divide\count255 by 60
  \multiply\count255 by -60
  \advance\count255 by \time
  \ifnum \count255 < 10 \xdef\klockan{\the\count1.0\the\count255}
  \else\xdef\klockan{\the\count1.\the\count255}\fi
\endgroup

\newcommand\indic[1]{\boldsymbol1\cpar{#1}}

\newcommand\dto{\stackrel{\rm d}{\to}}

\newcommand\eqd{\stackrel{\rm{d}}{=}}

\newcommand\ntoo{\ensuremath{{n\to\infty}}}
\newcommand\set[1]{\ensuremath{\{#1\}}}
\newcommand\bigset[1]{\ensuremath{\bigl\{#1\bigr\}}}

\newcommand\dd{\,\mathrm{d}}
\newcommand\ddx{\mathrm{d}}
\newcommand\dx{\ddx x}
\newcommand\bbR{\mathbb R}

\newcommand\bbZ{\mathbb Z}

\renewcommand\P{\operatorname{\mathbb P{}}}
\newcommand\E{\operatorname{\mathbb E}{}} 
\newcommand\ga{\alpha}

\newcommand\gd{\delta}

\newcommand\gG{\Gamma}

\newcommand\eps{\varepsilon}
\newcommand\ttau{\widehat\tau}
\newcommand\xpar[1]{(#1)}
\newcommand\bigpar[1]{\bigl(#1\bigr)}
\newcommand\Bigpar[1]{\Bigl(#1\Bigr)}

\newcommand\cpar[1]{\{#1\}}

\newcommand\bigabs[1]{\bigl\lvert#1\bigr\rvert}

\newcommand\intoi{\int_0^1}
\newcommand\intoo{\int_0^\infty}
\newcommand\intoooo{\int_{-\infty}^\infty}
\newcommand\oi{\ensuremath{[0,1]}}

\newcommand\oooox{(-\infty,\infty]}
\newcommand\HH{\mathrm{H}}
\newcommand\Binomial{\textrm{Binomial}}
\newcommand\Poisson{\textrm{Poisson}}
\newcommand\Po{\Poisson}
\newcommand\Exp{\textrm{Exponential}}

\newcommand\Gumbel{\textrm{Gumbel}}

\newcommand\cI{\mathcal I}

\newcommand\cL{{\mathcal L}}

\newcommand\cN{\mathcal N}

\newcommand\cU{{\mathcal U}}

\newcommand\cX{{\mathcal X}}
\newcommand\cY{{\mathcal Y}}

\newcommand\dtv{d_{\mathrm{TV}}}

\newcommand\Op{O_{\mathrm p}}
\newcommand\bP{\overline P}
\newcommand\bF{\overline F}
\newcommand\htau{\widehat{\tau}}
\newcommand\ceil[1]{\lceil#1\rceil}
\newcommand\bigceil[1]{\bigl\lceil#1\bigr\rceil}

\newcommand\floor[1]{\lfloor#1\rfloor}
\newcommand\lrfloor[1]{\left\lfloor#1\right\rfloor}
\newcommand\rand{R}
\newcommand\taun{\tau_n}

\newcommand\sss[1]{_{(#1)}}
\newcommand\hXi{\widehat\Xi}
\newcommand\hXisn{\hXi_n[s_n]}
\newcommand\hXiisn{\hXi_n'[s_n]}

\newcommand\tPi{\widetilde\Pi}
\newcommand\tPix{\tPi}
\newcommand\tPii{\widetilde\Pi'}
\newcommand\tXi{\widetilde\Xi}
\newcommand\hPi{\widehat \Pi}
\newcommand\qw{^{-1}}

\newcommand\etto{\bigpar{1+o(1)}}

\newcommand\rhs{right-hand side}
\newcommand\frax[1]{\{#1\}}
\newcommand\ii{\mathrm{i}}

\newcommand\qx{\chi}
\newcommand\hq{\widehat q}
\newcommand\hqx{\widehat{\qx}}

\newcounter{CC}
\newcommand{\CC}{\stepcounter{CC}\CCx} 
\newcommand{\CCx}{C_{\arabic{CC}}}     
\newcommand{\CCreset}{\setcounter{CC}0} 
\newcounter{cc}
\newcommand{\cc}{\stepcounter{cc}\ccx} 
\newcommand{\ccx}{c_{\arabic{cc}}}     
\newcommand{\ccdef}[1]{\xdef#1{\ccx}}     
\newcommand{\ccname}[1]{\cc\ccdef{#1}}    
\newcommand{\ccreset}{\setcounter{cc}0} 

\newcommand\gax{\gamma}
\newcommand\gaa{\alpha}
\newcommand\tb{a}
\newcommand\rmH{{\rm H}}

\endlocaldefs

\begin{document}

\begin{frontmatter}
\title{Maximal Counts in the Stopped Occupancy Problem}
\runtitle{Stopped Occupancy Problem}

\begin{aug}
\author[A]{\fnms{Alexander}~\snm{Gnedin}\ead[label=e1]{a.gnedin@qmul.ac.uk}},
\author[B]{\fnms{Svante}~\snm{Janson}\ead[label=e2]{svante.janson@math.uu.se}\orcid{0000-0002-9680-2790}}
\and
\author[C]{\fnms{Yaakov}~\snm{Malinovsky}\ead[label=e3]{yaakovm@umbc.edu}\orcid{0000-0003-2888-674X}}
\address[A]{School of Mathematical Sciences, Queen Mary University of London \printead[presep={ ,\ }]{e1}}

\address[B]{Department of Mathematics, Uppsala University \printead[presep={,\ }]{e2}}

\address[C]{Department of Mathematics and Statistics, University of Maryland, Baltimore County \printead[presep={,\ }]{e3}}
\end{aug}

\begin{abstract}
We revisit a version of the classic occupancy scheme, where
balls are thrown until almost all boxes receive a given number of balls.
Special cases are widely known as coupon-collectors and dixie cup problems.
We show that as the number of boxes tends to infinity, the distribution of the maximal occupancy count
does not converge, but can be  approximated by a convolution of two Gumbel
distributions,
with the approximating distribution having oscillations close to periodic on a logarithmic scale.
We pursue two approaches: one relies on lattice point processes obtained by  poissonisation of the number of balls and boxes, and the other employs interpolation of the multiset of occupancy counts to a point process on reals.
This  way we gain considerable insight in known asymptotics obtained previously by mostly analytic tools.
Further results concern the moments of maximal occupancy counts and ties for the maximum.
\end{abstract}



\end{frontmatter}

\section{Introduction}
\noindent
In the classic sequential occupancy scheme balls are thrown independently in
$n$ boxes, with
each ball landing with equal probability $1/n$ in each box.
The allied waiting time problems concern the distribution of the random number of trials required to satisfy  specified occupancy conditions for boxes \cite{JK}.
Technically, the random variables in focus here are stopping times that terminate the allocation process by way of an adapted  nonanticipating rule.
From the early days of probability theory such questions attracted much attention and were studied under different guises.
For instance,  a  problem treated in de Moivre's seminal treatise  \cite{DeMoivre} asks one
`to find in how many trials [a gambler] may with equal chance [i.e., probability about $1/2$] undertake  with a pair of common dice to throw all the doublets'.
A  better known   modern textbook example  is the  coupon collector's problem (CCP), which  in terms of the occupancy scheme deals
with  the number of balls thrown until no empty boxes are left.
The model where the balls are thrown until every box contains
more than $m$ balls is sometimes called the dixie cup problem \cite{Ilienko}.

In the present paper we are not primarily interested in the stopping time
itself, i.e., the required number of balls, but rather
in the configuration when the stopping occurs.
More precisely, we study the largest number of balls in a box
at that time, and also the second largest, and so on.
 Ivchenko in a series of papers
\cite{Ivchenko2,Ivch1, Ivch2}
studied a wide range of  statistics of the sequential occupancy scheme terminated by a stopping time;
a summary of  some of his work  is found in  \cite{IvSu}.
The present note is inspired by his result on the maximum box occupancy count $M_n$ observed when the stopping occurs as in the CCP  or according
to a more general criterion
(this appeared as a special case  of \cite[Theorem 9]{Ivch1} and
 \cite[Theorem 2]{Ivch2}).
The  distribution of $M_n$ does not converge as $n\to\infty$  even after
suitable normalisation; instead the approximating asymptotic distribution
oscillates on a logarithmic
 scale.
This is a common and well understood phenomenon when considering asymptotics of discrete random variables that do not require
scaling due to  bounded variances.
The asymptotic result can formally be stated either as convergence in
distribution of suitable subsequences, or (equivalently)
as approximation in total variation distance with
some family of random variables.
A simple but typical example of this kind is the distribution of  the maximum of $n$ i.i.d.\ geometric random
variables  (see  for instance \cite[Example 4.3]{SJ175}, also see  \refL{LeLem2} in the sequel for a general framework).

In this paper we employ the familiar device  of embedding  the occupancy scheme in a
Poisson process, to link directly properties  of $M_n$ with  the extreme value theory \cite{Mikosch, Resnick}.
This way  we identify
the asymptotic distribution with a discretised convolution of two Gumbel
distributions (\refT{TI2}).
The benefits of  the poissonisation of the number of balls in the occupancy problems have been long known
\cite{GHP, Karlin, Kolchin}.
A novel element in our proof
 is the use of {`bi-poissonisation'}, which amounts to also replacing fixed number of boxes $n$ by random, thus achieving independence of multiplicities   of occupancy counts for each given time.
Moreover,  the time evolution of the occupancy model as a whole becomes similar to the processes familiar from studies in population dynamics and queueing theory,
which sheds a new light on the  processes of small counts (Section \ref{Sdyn}).
A competing approach we  pursue here  (Section  \ref{Spf})  relies on interpolating  the multiset of integer
occupancy counts to a point process on $\bbR$,
then showing that it can be approximated
by a Poisson process with a suitable (exponential) rate, with all points
rounded to integers.
Extending the result about $M_n$ we show that a similar approximation holds  for
any given number of  maximal order statistics of the box occupancy counts
(\refT{stoppedmax}).

We also discuss (Section  \ref{Smom}) the asymptotics of moments  for $M_n$ and other stopped maximal order statistics of occupancy counts.
Complementing  the previous studies \cite{Ivch0, Ivch1, Ivch2, Kolchin, VS},
we  obtain  these indirectly by means of  new exponential tail estimates, which might be of  independent interest for other contexts related  to the extreme value theory.

Sampling from a discrete distribution outputs repeated values.
Criteria are known to ensure that  ties for  the maximum do not vanish asymptotically \cite{BES}.
In particular,  for the geometric source distribution the probability of such a tie is
known to  undergo tiny fluctuations on the logarithmic  scale \cite{Brands, KP}.
In this direction, we  examine (Section  \ref{M-max}) the multiplicity of $M_n$, and show
that the fluctuations turn even smaller due to a smoothing effect caused by random stopping.

Following  \cite{Ivch1, Ivch2}, we consider the occupancy process stopped when there remain only  $\ell$  boxes
that contain at most $m$ balls each.
Thus the case $m=0$ and $\ell=0$ corresponds to the CCP.
The integer parameters $m\geq0$ and $\ell\geq 0$ will be fixed throughout the paper; many
variables below depend on them, but this will not be shown in the notation.
(The parameters  $\ell$ in  \cite{Ivch1} and $m$ in
\cite{Ivch2} can be allowed to vary with $n$, but we will only focus
on the case of fixed values where the results turn out to be most interesting.)

\subsection{\bf Notation}
For common probability distributions we use the self-explaining notation, for instance  Poisson$(t)$, Geometric$(\alpha)$.
For identity and convergence in distribution we write $\stackrel{\rm d}{=}$ and $\stackrel{\rm d}{\to}$, respectively.
The total variation distance between the distributions of random elements $X$ and $Y$  is denoted $d_{\rm TV}(X,Y)$.
The symbols $\lfloor ~ \rfloor, \lceil  ~\rceil$ and $\{ ~ \}$ denote, respectively,  the  floor, the ceiling  and the fractional part functions.
For shorthand we set $L:=\log n$ and $\log L:=\log\log n$.
 Unspecified limits and asymptotic relations
such as $o(1)$, $f_n\sim g_n$
(meaning $f_n/g_n\to 1$, or equivalently $f_n=g_n(1+o(1))$)
and $f_n\asymp g_n$ (meaning that $f_n=O(g_n)$ and $g_n=O(f_n)$)
all presume $n\to\infty$.
We say that some limit  holds uniformly in $x=o(f_n)$
 if it holds uniformly for all $x$ with $|x|\le g_n$ for any function
 $g_n=o(f_n)$; the uniformity in $x=O(f_n)$ is understood similarly.
$\Op(1)$ means bounded in probability (i.e., tight), and $o_{\rm p}(1)$ means convergence to $0$ in probability.
The term `with high probability'   (w.h.p.)\ will mean that a certain event has probability converging to one as $n\to\infty$,
and `almost surely' (a.s.) will mean an event of probability one.

\section{The poissonised occupancy scheme}\label{S2}
\noindent
We will be dealing with
a standard continuous-time version of the discrete sequential
occupancy scheme.
Let $(\Pi_i(t), t\geq 0), i\in{\mathbb Z}_{>0},$ be i.i.d.\ replicas of a
Poisson counting process $(\Pi(t), t\geq 0)$ with unit rate.
We interpret $\Pi_i(t)$ as the number of balls allocated  into  box $i$ by time $t$;
this box occupancy count appears in the literature under different names such as score or load, to mention a few.

The problem with $n$ boxes concerns the occupancy counts $\Pi_i(t)$ for $i\in[n],$ where  $[n]:=\{1,\dots,n\}$.
The aggregate arrival process to the batch of $n$ boxes is Poisson with rate
$n$. Given such an arrival  occurring at  time $t$,
the ball is allocated into each of the
$n$ boxes with the same probability $1/n$,
independently of the past allocations.
Thus the random sequence of states of the set of $n$ boxes follows the same
dynamics as in the classic discrete-time occupancy scheme studied in detail in \cite{Kolchin}.
But keep in mind that  the time scale differs
 from the discrete case in that the mean number of balls dropped in $t$ units of continuous time is $nt$.

The advantage of the poissonised model is the exact independence among the boxes, whereas in the discrete-time occupancy scheme the independence
only holds asymptotically for large number of balls  \cite{Karlin, Kolchin}.
Moreover,
the adopted setting with infinitely many  $\Pi_i$'s allows one to consistently define  the occupancy schemes
for all $n$ on the same probability space.

For  point and cumulative  probabilities of  the Poisson distribution we use
the notation
\begin{align}
  p_r(t)=e^{-t}\frac{t^r}{r!},~
 P_r(t)=\sum_{i= 0}^r p_i(t), ~
\overline{P}_r(t)=1-P_r(t),
\end{align}
where $r\in{\mathbb Z}_{\geq0}, t\geq0$.
These are further related via the standard formulas
\begin{eqnarray}\label{Pp}
\overline{P}_r(t)&=&p_r(t)\sum_{i=1}^\infty \frac{t^i}{(r+1)\cdots(r+i)}, \\
\label{ingamma}
  P_r(t)&=&\int_t^\infty p_r(s){\rm d}s,
\end{eqnarray}
where (\ref{ingamma}) connects Gamma  and Poisson distributions.
If $t,r \to \infty$ so that $\limsup t/r<1$, then from (\ref{Pp}) follows the asymptotics
\begin{equation}\label{pPas}
\overline{P}_r(t)\sim  \frac{t}{r-t}\,p_r(t).
\end{equation}

Let, as in the Introduction,
$m\geq0$ and $\ell\geq0$ be fixed  integers.
Define $\tau_n$ to be the first time when  there remain only  $\ell$ out of
 $n$ boxes
that contain at most $m$ balls each, that is
\begin{align}\label{taun}
\tau_n:=\min\{t\geq0:  \Pi_i(t)> m {\rm~for  ~all~but~} \ell~{\rm indices~}i\in[n]\}
\end{align}
(we assume $n> \ell$ to ensure $0<\tau_n<\infty$ a.s.).
In particular, for $m=0,\ell=0$ this   is the  time when each of the $n$ boxes becomes occupied by at least one ball,
which is the termination condition in the CCP.
The case $\ell=0, m\geq1$ aligns with the dixie cup problem.

Since the Poisson distribution is discrete, repetitions among the occupancy counts $\Pi_i(t)$
occur with positive probability.
Let $M_{n,1}(t)\geq\cdots\geq M_{n,n}(t)$ be the nonincreasing sequence of  order statistics of  $\Pi_1(t),\ldots,\Pi_n(t)$.
The order statistics capture the allocation of some random number of indistinguishable balls into  $n$ indistinguishable boxes.
From a combinatorial viewpoint, this  can be regarded as  a generalised partition of an integer into $n$ parts which can be zero.
The same data can be equivalently encoded
in the sequence of {\it multiplicities}
\begin{align}
\mu_{n,r}(t):=\#\{i\in [n]: M_{n,i}(t)=r\}, ~r\in{\mathbb Z}_{\geq0},
\end{align}
that count repetitions among the box occupancy counts.
In particular, $\mu_{n,0}(t)$ is the number of empty boxes.
By independence among the boxes
\begin{align}
\mu_{n,r}(t)\stackrel{\rm d}{=}{\rm Binomial}(n, p_r(t)),
\qquad r\in{\mathbb  Z}_{\geq 0},
\end{align}
and the joint distribution of the multiplicities is multinomial with infinitely many classes.
The study of  multiplicities for large number of balls and boxes is
a central theme in the occupancy problems \cite{Kolchin}.

The equivalence of the two descriptions of the allocation of balls is
established by the relations
\begin{align}
M_{n,i}(t)\leq k \Longleftrightarrow \sum_{j=1}^\infty \mu_{n,k+j}(t)\leq i-1
\end{align}
for $i\in[n], k\in{\mathbb Z}_{\geq 0}$. Specifically, for the largest box occupancy count we have
\begin{align}
M_{n,1}(t)=\max\{r :\mu_{n,r}(t)>0\}.
\end{align}

With these notations, the number of balls thrown by
time $t$ has the threefold representation
\begin{align}
\sum_{i=1}^n\Pi_i(t)=\sum_{i=1}^n M_{n,i}(t)=\sum_{r=0}^\infty r \mu_{n,r}(t),
\end{align}
Note also that the random variables  of our primary interest
are the largest few occupancy numbers when we stop
(which are the same in the discrete-time and continuous-time models),
that is
$M_{n,i}(\tau_n)$'s with $i$ less than some bound not depending on $n$.
We therefore use the shorthand notation
\begin{align}\label{Mni}
M_{n,i}:=M_{n,i}(\tau_n).
\end{align}
We also sometimes use the even shorter notation
\begin{align}\label{Mn}
M_n:=M_{n,1}=M_{n,1}(\tau_n)
\end{align}
for the maximum box occupancy count when the allocation is stopped.

\section{Preliminaries on point processes}\label{S3}
\noindent
Throughout we will be exploiting basic facts on Poisson and related point
processes as found in many excellent texts \cite{Kallenberg-rm, Kingman,
LastPenrose, Resnick}.
In this section we remind the bare minimum, also using this opportunity to introduce  constructions
needed in later sections.

A point process  $\rm H$ on $\mathbb R$  (or a more general  Polish space) is a random  locally finite   Borel measure with values in ${\mathbb Z}_{\geq0}$.
Such point process
can be represented as a (finite or infinite) sum
\begin{align}\label{PP}
  \rmH=\sum_i\gd_{\eta_i}
\end{align}
of Dirac masses at random points $\eta_i$.
We will sometimes with a minor abuse of notation write this
as $\rmH=\set{\eta_i}$, thus identifying
the point process
with the multiset $\set{\eta_i}$ of its points, that is its  atoms endowed
with some  multiplicities.
(We prefer the measure-theoretic term `atom' to make difference with
nonrandom points of the background space.)
The intensity measure of $\rm H$ is the function that assigns ${\mathbb E}[{\rm H}(A)]$
to each Borel set $A$.
We  call $\rm H$  simple if the multiplicity of each atom is one almost surely.

Recall that a Poisson point process is a point process $\HH$ where $\HH(A)$ (the number
of points in $A$) has a Poisson distribution for every Borel set $A$, and
these numbers for different, disjoint sets
are independent.

\subsection{Transformations of Poisson processes}\label{S3.1}
The unit rate Poisson process $\Pi$ on ${\mathbb R}_+$ is a simple point process  whose atoms are representable as sums
$\eta_i=E_1+\cdots+E_i$, where the terms
are independent  with standard exponential distribution, so the $i$th atom has  Gamma$(i,1)$ (aka Erlang) distribution satisfying
\begin{align}\label{sa1}
{\mathbb P}[\eta_i\leq t]
=\bP_{i-1}(t),
\qquad t\geq0.
\end{align}
The intensity measure of $\Pi$ is the Lebesgue measure on the halfline.

Inhomogeneous Poisson point processes are  uniquely characterised by their intensity measures.
Such  processes on $\mathbb R$  can be constructed from $\Pi$ by  the  measure-theoretic pushforward  $\Pi\circ f^{-1}$,
which is implemented through transporting the atoms with  function $f$;
thus the multiset $\set{\eta_i}$ is mapped to $\set{f(\eta_i)}$.
The intensity measure  of $\Pi\circ f^{-1}$  is the pushforward of the Lebesgue measure by $f$.

An important role in the extreme-value theory  is played by the {\it exponential} Poisson process $\Xi$
on $\mathbb R$,  obtained as  pushforward of $\Pi$ by $f(t)={-\log t}$, thus
with the intensity  measure
\begin{align}\label{EXi}
\E[\Xi({\rm d}x)]=e^{-x}{\rm d}x,~ x\in {\mathbb R}.
\end{align}
(The name is not common but has been used in the literature,
see \cite[Section 6.2.2]{Berestycki}.)
The atoms of
$\Xi$ comprise a decreasing to $-\infty$ sequence
of  random variables  $\xi_i:=-\log\eta_i$,
with
distribution,
by \eqref{sa1},
\begin{align}\label{sa2}
{\mathbb P}[\xi_i\leq x]=P_{i-1}(e^{-x}), ~~~x\in {\mathbb R},
\end{align}
and thus with density, see \eqref{ingamma},
\begin{align}\label{sa3}
e^{-x}p_{i-1}(e^{-x}), ~~~x\in {\mathbb R}.
\end{align}
In particular,
the largest atom $\xi_1$ of $\Xi$ has the (standard) Gumbel distribution
\begin{align}\label{gumbel}
{\mathbb P}[\xi_1\leq x]=e^{-e^{-x}}, ~~~x\in {\mathbb R}.
\end{align}
The best known instance $i=1$, as well as the distributions \eqref{sa2}
with $i>1$, were introduced in \cite{Gumbel},
which justifies the notation ${\rm Gumbel}(i)$
for the distribution \eqref{sa2} of $\xi_i$.
In this nomenclature the (standard) Gumbel distribution becomes ${\rm Gumbel}(1)$.
For each $i$ there is an associated  scale-location family of distributions.

We define
 $\Xi_{b}:=\Xi+b,~ b\in {\mathbb R},$ to be the  translation  of $\Xi$ with
atoms
$\xi_i+b$; these have shifted Gumbel$(i)$ distributions
\begin{align}
{\mathbb P}[\xi_i+b\leq x]=P_{i-1}(e^{-x+b}), ~~~x\in {\mathbb R}.
\end{align}
Thus, $\Xi_b=\set{\xi_i+b}$ is a Poisson process with intensity measure $e^{b-x}\dx$,
$x\in\bbR$.

\subsection{Lattice point processes} For  a point process with nonatomic intensity measure the probability of an atom occurring at fixed location is zero.
In contrast to that, a point process on  the integer lattice ${\mathbb Z}$ is in essence a two-sided random sequence of multiplicities at integer locations.
In particular,  a  lattice Poisson point process
is identifiable with an array of independent Poisson random variables with given parameters.

We denote by $\Xi^\uparrow_{b}$  the lattice counterpart of  $\Xi_{b}$
obtained by applying the ceiling function:
\begin{align} \label{xi+}
\Xi^\uparrow_{b}:=\set{\ceil{\xi_i+b}}.
\end{align}
We choose  rounding up (rather than down) to have the intensity measures
agreeing on semi-closed intervals $(-\infty,r]$, $r\in\bbZ$,
thus forcing  the distribution functions of the respective atoms to coincide at integers:
\begin{equation}\label{di-co-df}
{\mathbb P}[\lceil \xi_i+b\rceil \leq r]   = {\mathbb P}[ \xi_i+b\leq r]=P_{i-1}(e^{-r+b}), ~~~r\in {\mathbb Z}.
\end{equation}
In the case $i=1, b=0$ we obtain a natural discrete analogue of the Gumbel distribution.
The intensity measure of $\Xi^\uparrow_{b}$ is  supported by $\mathbb Z$,
with masses comprising a two-sided geometric sequence
\begin{align}\label{lxx1}
{\mathbb E}[\Xi^\uparrow_{b}(\{r\})]
=\E [\Xi_{b}(r-1,r]]
=\int_{r-1}^re^{b-x}\dd x
=e^{-r+b}(e-1),
\end{align}
hence with the right tail
\begin{align}\label{lxx1s}
{\mathbb E}[\Xi^\uparrow_{b}(r,\infty]]=
{\mathbb E}[\Xi^\uparrow_{b}[r+1,\infty]]
=\E [\Xi_{b}(r,\infty]]
=e^{-r+b}.
\end{align}


\subsection{The marking theorem and the Poisson shift}
In plain terms, a basic version of  the marking theorem   says that if $\eta_i$'s are random atoms of some Poisson point process $\rm H$ and random `marks' $\zeta_i$'s are i.i.d.\ independent of $\rm H$,
then the pairs $(\eta_i,\zeta_i)$ define a bivariate Poisson process in a product space. This proves very useful to construct other Poisson processes as transforms $\{f(\eta_i, \zeta_i)\}$.
In particular, if $\rm H$ is on $\mathbb R$ and $\zeta_i$'s are real-valued, then the pairwise sums $\eta_i+\zeta_i$ are atoms defining another  Poisson process, whose intensity
measure is the convolution of the intensity measure of $\rm H$ and the distribution of $\zeta_1$.

For a general point process $\rm H$  on $\mathbb R$ and $h\geq0$
we define its
{\it Poisson shift} $T_h\circ{\rm H}$ as the above
operation with  atom-wise adding of  independent $\zeta_i\stackrel{\rm d}{=} {\rm Poisson}(h)$.
Note that if $\rm H$ has $k$ atoms at the same location $x$, then each of them contributes to $T_h\circ{\rm H}$
a unit mass at $x$ shifted by an independent Poisson variable.

Notably, on the exponential Poisson process $\Xi$ the Poisson shift acts
in distribution
like
 a deterministic translation
\begin{equation}\label{PoSh}
T_h\circ \Xi\stackrel{\rm d}{=} \Xi_{(e-1)h}.
\end{equation}
The proof  follows  by a simple calculation found in \cite[p.~153]{Berestycki} and the formula ${\mathbb E}[e^{\zeta_1}]=e^{h(e-1)}$.
Applying $T_h$ to the lattice process  $\Xi^\uparrow$ yields a distributional  copy of  $\Xi^\uparrow_{(e-1)h}$, as is clear from (\ref{PoSh}) since $\zeta_i$'s are integer-valued.

The  {\it Poisson flow} with initial state $\rm H$  is the Markov measure-valued process $(T_h\circ{\rm H}, ~h\geq0)$, where each atom independently of  the others undergoes unit jumps to the right
at the unit rate.
If $\rm H$ is a Poisson point process then so is also every $T_h\circ{\rm H}$.
See the recent monograph \cite{Dorogovtsev} for  the general theory of measure-valued processes.

\subsection{Mixed binomial point processes}
A Poisson point process with finite
intensity measure has a random sum representation
\begin{align}
{\rm H}=\sum_{i=1}^N \delta_{\eta_i},
\end{align}
where the random variables
 $\eta_1,\eta_2,\ldots$ are i.i.d., and $N$ is an independent from
 $\eta_i$'s Poisson random variable whose parameter is equal to the total
 mass of the intensity measure.

The latter form is an instance of  the more general {\it mixed binomial} point process (MBPP) \cite{Kallenberg-rm}, where $N$ is allowed to have  arbitrary distribution
on ${\mathbb Z}_{\geq0}$.
Conditionally on $N=n$, such a MBPP is just  a scatter of $n$ i.i.d.\ random points.
A subprocess obtained by restricting ${\rm H}$ to  a Borel set   $A$ is  again a MBPP
\begin{align}
{\rm H}|_A=\sum_{i=1}^N \indic{\eta_i\in A}   \delta_{\eta_i}\stackrel{\rm d}{=}\sum_{j=1}^{\widehat{N}}    \delta_{\widehat{\eta}_j},
\end{align}
where ${\widehat{N}}$ has a mixed binomial distribution ${\rm Binomial}(N, \alpha)$ with $\alpha={\mathbb P}[\eta_1\in A]$, and
$\widehat{\eta}_j$'s are i.i.d.\ with distribution
${\mathbb P}[{\widehat{\eta_1}}\in B]= {\mathbb P}[{{\eta_1}}\in B\mid \eta_1\in A]$.
If $N$ has a Poisson (respectively,  binomial) distribution
then also ${\widehat{N}}$   has a Poisson (respectively, binomial)  distribution.

For another  MBPP
\begin{align}
  {\rm H}'=\sum_{i=1}^{N'} \delta_{\eta_i},
\end{align}
that only differs from $\rm H$ by the distribution of the total count, presuming
$N, N',\eta_i$'s defined on the same probability space, we will have $\{N=N'\}=\{{\rm H}={\rm H}'\}$. Therefore, for the total variation distance
we have the identity
\begin{equation}\label{dident}
d_{\rm TV}({\rm H},{\rm H}')=d_{\rm TV}(N,N'),
\end{equation}
which follows from the definition of the distance
as the infimum of the non-coincidence probability ${\mathbb P}[{\rm H}\neq{\rm H}']$ taken over all couplings.

\subsection{Weak convergence}\label{SSvague}
Weak convergence (convergence in distribution)
of point processes is defined using  the `vague topology'
in the space of locally finite measures; this means roughly convergence
 in the weak topology of restrictions on compact subsets.
Note that
convergence in the weak topology on $\bbR$ is defined only for
a.s.\ finite point processes with an a.s.\ finite limit process,
which is a situation not applicable here. Instead,
the point processes treated in this paper are point processes $\HH$ on $\bbR$
such that ${\rm H}(x,\infty)<\infty$ a.s.\ for every $x\in {\mathbb R}$;
for such point processes
 we may use the representation ${\rm H}=\sum_{i=1}^\infty\delta_{\eta_i}$
 with $\eta_1\geq\eta_2\geq\ldots$,
setting formally $\eta_i=-\infty$ if the sequence of atoms is finite.
Note that $\eta_i\to-\infty$ as $i\to\infty$ since point processes are
assumed to be locally finite, that is finite on every compact set.

A technical note is that a point process on $\bbR$ such that
${\rm H}(x,\infty)<\infty$ a.s.\ for every $x\in {\mathbb R}$,
also can be regarded as a point process on $\oooox$;
we will not actually put any atoms at $+\infty$, but  the inclusion of
$+\infty$ in the background space increases the family of compact sets (for example,
$[x,\infty]$ becomes compact), which makes the vague topology stronger.
We will use this
one-point closure and note that
convergence of point processes in the vague topology on $\oooox$
is equivalent to the finite dimensional convergence of the ordered sequence of atoms,
as shown by the following lemma
(cf.\ \cite[Lemma 4.4]{SJ136} with a trivial change of variables
$\oooox\to(0,\infty]$).

\begin{lemma}\label{L:point1}
Suppose that\/ $\HH_n$, $1\le n\le\infty$, are point processes on  $\oooox$,
and write $\HH_n=\set{\eta_{n,i}}_{i=1}^{N_n}$ with
$\eta_{n,1}\ge\eta_{n,2}\ge\dots$ and $0\le N_n\le\infty$.
If some $N_n<\infty$, define further $\eta_{n,i}=-\infty$
for $i>N_n$.
Then $\HH_n\stackrel{\rm d}{\to}\HH_\infty$,
in the vague topology for measures on $\oooox$,
if and only if
$(\eta_{n,1},\eta_{n,2},\dots)
\stackrel{\rm d}{\to}(\eta_{\infty,1},\eta_{\infty,2},\dots)$
in the standard sense that all finite dimensional distributions converge.
\qed
\end{lemma}

%

\section{Stopped maximum}\label{S4}
\noindent
Recall the notation $M_n:=M_{n,1}(\tau_n)$
for the maximum box occupancy count observed as the allocation is stopped at time $\tau_n$.
The distribution of $M_n$ was studied by Ivchenko in the framework of the
classic discrete-time occupancy scheme.  We here study its distribution using the poissonised continuous-time scheme.

\subsection{Stopping time}\label{StopT}
Let $m\ge0$ and $\ell\ge0$ be fixed, and consider the stopping time
$\tau_n$ in \eqref{taun},
which can be alternatively defined through the multiplicities  as
\begin{align}\label{taun1}
\tau_n=\min\{t: \mu_{n,0}(t)+\cdots+\mu_{n,m}(t)=\ell\},
\end{align}
 where we assume $n>\ell$.

First,
the distribution of the stopping time  follows readily by identifying  $\tau_n$ with the $(\ell+1)$st last time when one of $n$ boxes receives its $(m+1)$st ball.
The distribution of  the $(m+1)$st arrival time to a particular box is ${\rm Gamma}(m+1,1)$,
 whence by independence among the boxes
\begin{equation}\label{taund}
{\mathbb P}[\tau_n\in {\rm d}t]= n \binom{n-1}{\ell} (P_{m}(t))^\ell (\overline{P}_{m}(t))^{n-\ell-1} p_{m}(t){\rm d}t, ~~~t\geq0.
\end{equation}
In greater detail,  the event defining  $\tau_n$  occurs when one box  receives its $(m+1)$st  ball, $\ell$ boxes contain at most $m$ balls each, and the
remaining  $n-\ell-1$ boxes contain at least $m+1$ balls each.
 By exchangeability of the boxes,
the distribution of $M_n$ conditional on $\tau_n$ will not change   if we
also condition on the indices of the boxes that satisfy the said constraints.

This implies   the following fact noticed in \cite[Lemma 1]{Ivch1}.
\begin{fact}\label{FACT}
Conditioned on $\tau_n=t$,
the subsequence of $n-\ell-1$ stopped occupancy numbers of  boxes that received at least $m+1$ balls (strictly) before time $t$
 is  i.i.d.\ and
independent of the complementing subsequence (also i.i.d.)\
of $\ell$ boxes with at most $m$ balls.
Moreover,
 both subsequences have
 truncated\ {\rm Poisson}$(t)$  distributions:
the first one on $\{m+1, m+2,\ldots\}$ and the second  on $\{0,\ldots,m\}$.
\end{fact}

For $t\to\infty$  the truncated Poisson distribution on $\{0,\ldots,m\}$ converges to  the Dirac measure $\delta_m$, because
$p_{r+1}(t)/p_r(t)=t/(r+1)\to\infty$.  Since
$\tau_n\stackrel{\rm p}{\to}\infty$, w.h.p.\
 the number of boxes with exactly $m$ balls immediately before stopping is $\ell+1$
and no box contains lesser number of balls, whence by monotonicity in the stopping condition in (\ref{taun1})
\begin{equation}\label{taun2}
\tau_n=\sup\{t: \mu_{n,m}(t)=\ell+1\}\qquad{\rm w.h.p.}
\end{equation}
(where $\sup\varnothing:=\infty$),
which gives yet another, asymptotic, interpretation of $\tau_n$ in terms of  the sole multiplicity $\mu_{n,m}(t)$.
The relation (\ref{taun2}) holds a.s.\ for $m=0$.

We emphasise (\ref{taun2}) to match the temporal domain we need here with a classification of
asymptotic regimes for the growing number of balls in \cite{Kolchin}.
In the terminology of this book,
the range of $\tau_n$ is the {\it right $m$-domain}, which for the poissonised model can be characterised by the properties
$\mu_{n,m}(t)=O_{\rm p}(1)$,  $\mu_{n,r}(t)=o_{\rm p}(1)$ for $r<m$, and  $\mu_{n,r}(t)\stackrel{\rm p}{\to}\infty$ for $r>m$.

\subsection{Approximating  the distribution of $M_n$}
We turn to the stopped maximum.
By  the virtue of the adopted stopping condition
we cannot have $M_n<m+1$, and for  values $r \geq m+1$ conditioning on the indices yields
\begin{align*}
&{\mathbb P}
[M_n\leq r\, |\,\tau_n=t]=
\qquad\qquad\qquad\qquad\qquad\qquad\qquad\qquad\qquad\qquad\qquad\qquad\qquad\qquad\notag\\&
{\mathbb P}[M_n\leq r\, |\,
\Pi_n(t-)<\Pi_n(t)=m+1
{\rm~ and~}  \Pi_i(t)\le m < \Pi_j(t)~{\rm for~}1\leq i\leq \ell<j\leq n-1].
\end{align*}
Since the event defined by the condition entails $M_n=
\max(\Pi_{\ell+1}(t),\ldots,  \Pi_{n-1}(t))$
 (if $n>\ell+1$),
by independence the last formula becomes
\begin{align}
{\mathbb P}[M_n\leq r \,|\,\tau_n=t]=\prod_{j=\ell+1}^{n-1} {\mathbb P}[\Pi_j(t)\leq r\,|\, \Pi_j(t)\geq m+1]=\left(1-\frac{\overline{P}_r(t)}{\overline{P}_{m}(t)}\right)^{n-\ell-1}.
\end{align}
Integrating out the stopping time we obtain the unconditional distribution of the stopped maximum in the form of a mixture
\begin{equation}\label{distrM}
{\mathbb P}[M_n\leq r]=
\int_0^\infty   \left(1-\frac{\overline{P}_r(t)}{\overline{P}_{m}(t)}\right)^{n-\ell-1} {\mathbb P}[\tau_n\in{\rm d}t], ~~~~\qquad\qquad~~r\geq m+1.
\end{equation}
Finding the asymptotics of   (\ref{distrM}) requires approximating both the mixing distribution and the integrand.

The first part is a customary task from the extreme-value theory, which we
include for completeness.
The stopping time $\tau_n$ has the same distribution as the $(\ell+1)$st order statistic for $n$ i.i.d. ${\rm Gamma}(m+1,1)$ random variables.
The constant
\begin{equation}\label{an}
\gaa_n:=L +  m\log L -\log m!
\end{equation}
is a well known $o(1)$ approximation to the upper $1/n$-quantile of
 Gamma$(m+1,1)$, see  \cite[p.~156]{Mikosch}.
With this centering,
we have for any fixed $s\in\bbR$,
\begin{align}\label{mSt}
  P_{m}(\gaa_n+s) \sim \frac{L^{m}}{m!}e^{-\gaa_n-s}=\frac{1}{n}e^{-s}
\end{align}
and thus $\Binomial\bigpar{n,P_{m}(\gaa_n+s)}
\stackrel{\rm d}{\to}
\Poisson(e^{-s})$.
Hence,
\begin{align}
\hskip2em&\hskip-2em
  \P[\tau_n\le \gaa_n+s]=\P\left[\text{Binomial}(n,P_{m}(\gaa_n+s))\le\ell\right]
\to P_\ell(e^{-s}),
\end{align}
which means that
we have weak convergence
\begin{align}\label{taulim}
\tau_n-\gaa_n
\stackrel{\rm d}{\to}
\tau
\end{align}
to a random variable
$\tau$ with $\P[\tau\le s]=P_\ell(e^{-s})$
and thus by \eqref{sa2} $\tau\eqd\xi_{\ell+1}$ in the notation there;
in other words $\tau$ has
${\rm Gumbel}(\ell+1)$   distribution
with density  \eqref{sa3}, that is
\begin{equation}\label{ell-Gum}
{\mathbb P}[\tau\in {\rm d}s]=e^{-s} p_\ell(e^{-s})
=\frac{1}{\ell!}\exp\bigpar{-(\ell+1)s-e^{-s}}
,\qquad
s\in{\mathbb R}.
\end{equation}
For $\ell=0$ the  limit distribution of $\tau_n$ is  standard Gumbel, which is a well known result
in the context  of   CCP and the dixie cup problems
 \cite{Ilienko}.
Comparing with Section \ref{S3.1}, we see  that $\tau$ can be realised as the $(\ell+1)$st largest atom of the exponential Poisson process $\Xi$.

Identifying a  proper  norming for the integrand (hence  $M_n$) is a much more delicate matter, requiring
a bivariate approximation of the Poisson probabilities.
 To that end, we introduce
\begin{equation}\label{bt}
\tb_n:=eL+ \left((e-1)m-\frac{1}{2}\right)\log L - \log\left( (e-1)m!^{e-1}\sqrt{2\pi e}\right),
\end{equation}
where the choice of the constant term is the matter of convenience. Set further
\begin{equation}\label{bc}
b_n:=\lrfloor{\tb_n},
\qquad
  c_n:=\{\tb_n\}
\end{equation}
so $\tb_n=b_n+ c_n$ is the decomposition in integer and fractional parts.
The next lemma is our main technical tool,
giving  asymptotics of Poisson probabilities  in a vicinity of $\gaa_n$ and
$b_n$.
For the time being we may
ignore the connection $L=\log n$ and just treat $L$
as a large parameter.

\begin{lemma} \label{LSt} Let $t=t(L,u)\geq 0$ and  integer $r=r(L,v){\geq0}$ for large enough $L$ be given by
\begin{eqnarray}
t&:=&L +  m\log L + u,\\
r&:=&eL+ \left((e-1)m-\frac{1}{2}\right)\log L+v.
\end{eqnarray}
Then, as $L\to\infty$,
\begin{align}
 \overline{P}_r(t)&\sim \frac{1}{e-1}p_r(t),
\label{logP1}\\\label{logP2}
\log p_r(t)&=
\\\notag&\hskip-3em
-L + (e-1)u-v -   \log\sqrt{2\pi e} -\frac{1}{2}\log\left( 1+\frac{v}{eL}\right) - \frac{(eu-v)^2}{2(eL+v)}+O\left(\frac{\log L} {L^{1/2}} \right)
\end{align}
uniformly in $u,v$ within the range
\begin{eqnarray}
\frac{u}{L}+1&\ge&\varepsilon,\label{sec-a1}\\
|eu-v|&=&O(L^{1/2})
\label{sec-a2}
\end{eqnarray}
 for any fixed $\varepsilon>0$.
\end{lemma}

\begin{proof}
It follows readily from the assumptions \eqref{sec-a1}--\eqref{sec-a2} that $t,r\to\infty$ with
 $L=O(t\wedge r)$  and  $t/r\to 1/e$.
This
gives the asymptotics in (\ref{logP1}) by the virtue of (\ref{pPas}).

For \eqref{logP2},    Stirling's  formula yields
\begin{equation}\label{StiFo}
\log p_r(t) =-t+ r\log\left( \frac{et}{r}\right)-\log
\sqrt{2\pi}-\frac{1}{2}\log r
+O\left({r}^{-1}\right).
\end{equation}
To work out the second term of (\ref{StiFo}), write $et-r=eu-v+(m+1/2)\log L$ and note that for large $L$
the assumptions
\eqref{sec-a1}--\eqref{sec-a2} further imply   $(et-r)/r=O(L^{-1/2})$ along with
\begin{align}
e L+v>e L+   eu+O(L^{1/2})>\varepsilon  e L + O(L^{1/2}) >\varepsilon L.
\end{align}
With all this in hand we calculate by expanding the logarithm
\begin{align}
r\log\left( \frac{et}{r}\right)&
=et-r-\frac{(et-r)^2}{2r}+O\left(L^{-1/2}\right)
\\\nonumber
&=
eu-v+(m+1/2)\log L -\frac{(eu-v)^2}{2(eL+v)}+O\left( L^{-1/2}\log L\right).
\end{align}
Another logarithm expansion gives
\begin{align}
\log r=\log(eL)+\log\left(1+\frac{v}{eL} \right)+O(L^{-1}\log L).
\end{align}
Plugging the last two formulas in (\ref{StiFo}),  we obtain (\ref{logP2})   by careful bookkeeping.
\end{proof}

Applying (\ref{logP1}) with $L=\log n$ results in the tail asymptotics
\begin{eqnarray}\label{Tas}
\overline{P}_{b_n+k}(\gaa_n+s)&\sim&  n^{-1} e^{{(e-1)s+c_n -k}},
\end{eqnarray}
 uniformly in $s$ and $k$ satisfying
$s\ge-(1-\eps)L$,
and $|es-k|=o(L^{1/2})$,
 and $k=o(L)$,
which hold in particular if  $s,k$ assume values in a bounded range.
For times $t=\gaa_n+s$
in  (\ref{distrM})  we  have   $P_{m}(t)\to 0$ locally uniformly in $s$.
Making use of  (\ref{Tas})  we  approximate the  integrand to arrive at the following theorem.

\begin{theorem}\label{TI}
For every fixed $k\in {\mathbb Z}$, as $n\to\infty$,
\begin{align}\label{IvTh}
{\mathbb P}[M_n-b_n\leq k]&=\int_{-\infty}^\infty p_\ell(e^{-s}) e^{-s} \exp\left(- e^{(e-1)s+c_n-k}\right)
{\rm d}s+o(1)
\\\nonumber
&= \frac{1}{\ell!}
\int_{-\infty}^\infty\exp\left(- e^{-s} -(\ell+1)s- e^{(e-1)s+c_n-k}\right)
{\rm d}s+o(1).
\end{align}
\nopf
\end{theorem}
A peculiar  feature of this result
(which is a common
phenomenon for asymptotics of discrete random variables that do not require
scaling)
is that the distribution defined by
(\ref{IvTh}) has, asymptotically, periodic fluctuations on the $\log n$
scale.
Thus, weak convergence of the centered stopped maximum only holds along
certain subsequences $(n_j)$ and
there are different possible limit distributions;
more precisely, we see from \eqref{IvTh} (or \eqref{dist-Z} below) that
convergence in distribution holds for any subsequence such that $c_n$ converges.
In  \cite{Ivch1}, equation  (\ref{distrM})  was manipulated using the change
of variable $z=nP_{m-1}(t)$, which lead
to a representation    equivalent to (\ref{IvTh}) via  the substitution
$z=e^{-s}$.

The following equivalent reformulation of
 Theorem \ref{TI}
in terms of  the total variation distance is
 new to our knowledge.
\begin{theorem}\label{TI2}
Let $b_n$ and $c_n$ be given by \eqref{bt}--\eqref{bc}.
Then
\begin{align}\label{ti2}
  \lim_{n\to\infty}d_{\rm  TV}(M_n-b_n,Z_n)=0
\end{align}
with
\begin{equation}\label{dist-Z}
Z_n: =\lceil \xi_1 + (e-1)\tau+c_n\rceil,
\end{equation}
where $\xi_1$ is standard Gumbel-distributed and independent of
$\tau\stackrel{\rm d}{=}{\rm Gumbel}(\ell+1)$ given by \eqref{ell-Gum}.
\end{theorem}

Note that the distribution of $Z_n$ depends on $n$ through the constant
$0\leq c_n\leq 1$, which is the source of the oscillations.
Recall also that the variables $M_n$ depend on the fixed parameters
$m\ge0$ and $\ell\ge0$;
in \refT{TI2}, $b_n$ and $c_n$ depend on $m$, and $\tau$
depends on $\ell$.
In the special case $\ell=0$ (CCP and dixie cup problem), $\xi_1$ and $\tau$ are
two independent standard Gumbel variables.

\begin{proof}
   Using \eqref{ell-Gum} and \eqref{gumbel}, we can write the first integral
   in \eqref{IvTh} as
   \begin{align}\label{no1}
\hskip2em&\hskip-2em
     \intoooo\exp\bigpar{-e^{(e-1)s+c_n-k}}\P\left[\tau\in\ddx s\right]
= \intoooo\P\left[\xi_1 \le -(e-1)s-c_n+k\right]
\P\left[\tau\in\ddx s\right]
\\\notag&
=\P\left[\xi_1 \le -(e-1)\tau-c_n+k\right]
=\P\left[\xi_1 +(e-1)\tau+c_n\le k\right]
\\\notag&
=\P\left[ Z_n\le k\right],
   \end{align}
using also \eqref{dist-Z}. Hence \eqref{IvTh} yields
\begin{align}\label{no2}
\P[M_n-b_n\le k] = \P[Z_n\le k]+o(1), \qquad\text{for every }k\in\bbZ.
\end{align}
It is easy to see, using that the sequence $(c_n)$ is bounded,
that \eqref{no2} is equivalent to \eqref{ti2},
see  \cite[Lemma 4.1]{SJ175} (or Lemma \ref{LeLem2} in what follows).
\end{proof}

We have stated \refT{TI2} with (roughly) centered variables $M_n-b_n$.
The result can also be stated without centering as follows.

\begin{theorem}\label{TI3}
  Let\/ $\tb_n$ be given by \eqref{bt}.
Then
\begin{align}\label{ti3}
  \lim_{n\to\infty}d_{\rm  TV}(M_n,\lceil \xi_1+ (e-1)\tau+\tb_n\rceil)=0
,\end{align}
where $\xi_1$ is standard Gumbel-distributed and independent of
$\tau\stackrel{\rm d}{=}{\rm Gumbel}(\ell+1)$ given by \eqref{ell-Gum}.
\end{theorem}
\begin{proof}
We have $d_{\rm  TV}(M_n-b_n,Z_n)=d_{\rm  TV}(M_n,Z_n+b_n)$,
and since $b_n$ is an integer,
\begin{equation}\label{ww1}
Z_n+b_n =\bigceil{\xi_1 + (e-1)\tau+c_n+b_n}
=\bigceil{\xi_1 + (e-1)\tau+\tb_n},
\end{equation}
using \eqref{bc}.
Thus \eqref{ti3}
is  an immediate consequence of \refT{TI2}.
\end{proof}

\section{Maximal occupancy counts  at fixed and random times}

\noindent
For the classic occupancy scheme
the multivariate asymptotics of order statistics  were explored in \cite{Ivch0, VS} by
the method of moments.
The   approach in \cite{Kolchin} relies on poissonisation and a conditioning  relation
connecting to the multinomial distribution.
This work revealed that the maximum occupancy count may exhibit different asymptotic  behaviours depending
  on  how the number of balls compares to $n$.
Parallel studies  in the extreme-value theory  confirmed that
 the maximum of $n$ Poisson variables with fixed parameter
 is asymptotically degenerate  \cite{Anderson70} and later
focused on an instance of the triangular  scheme, where
a continuous Gumbel limit exists \cite{Anderson}.
In this section we contribute to the past development by treating the occupancy scheme in the  framework of point processes and introducing a simple tool to deal with  dependence
among the multiplicities.

As in the previous section, we deal with the time range $t=\gaa_n+  O(1)$  (corresponding to the right $m$-domain recalled in Section \ref{StopT}), that is
$t=\gaa_n+s$, where
$\gaa_n$ is defined by \eqref{an} and
$s$ is allowed to vary in a large but bounded range.
With the norming
suggested by Lemma \ref{LSt}, we introduce a lattice point process ${\rm M}_n^s$ with $n$ atoms
\begin{align}
M_{n,i}(\gaa_n+s)-b_n,  ~i\in [n],
\end{align}
each corresponding to some box occupancy count $\Pi_j(\gaa_n+s)$, $j\in [n]$.
Representing via multiplicities,
\begin{align}
  {\rm M}_n^s:=\sum_{k=-\infty}^\infty  \mu_{n,b_n+k}(\gaa_n+s)\delta_k.
\end{align}

Intuitively,  the point process ${\rm M}_n^s$ captures a few maximal
occupancy counts, which are rare among the $n$ boxes for the temporal regime
 in focus.
An  obstacle on the way of approximating
${\rm M}_n^s$ for large $n$  is the dependence among the multiplicities, which persists in the poissonised scheme
through  the identity $\sum_r  \mu_{n,r}(t)=n$.
 Our
strategy to circumvent the dependence makes use of replacing the fixed
number of boxes with a random number
to pass to a Poisson point process $\widehat{\rm M}_n^s$, then compare $\widehat{\rm M}_n^s$ with a suitable exponential  Poisson process $\Xi_{b}$.

\subsection{Bi-poissonisation}  We randomise the poissonised occupancy scheme
by introducing    an auxiliary unit-rate Poisson point process ${\rm N}=\{\theta_i\}$
with atoms $\theta_1<\theta_2<\ldots$, independent
of $\Pi_1,\Pi_2,\ldots$.
We associate with $\theta_i$ the $i$th box, and treat  the whole arrival process  $\Pi_i$ to this box as a random mark attached to $\theta_i$.
Geometrically, the now bivariate point data can be plotted in the positive quadrant of the $(\theta, t)$-plane, by first erecting  a vertical  line  at each site  $\theta_i$ of the $\theta$-axis, then
populating  line $i$ with    the atoms of $\Pi_i$.
(The resulting planar point process has the intensity measure ${\rm d}\theta{\rm d}t$  but is not Poisson.)
In this picture, the bi-poissonised
occupancy scheme with size parameter $n$ corresponds to the configuration of atoms in a vertical strip  of width $n$.

For shorthand we write $N={\rm N}[0,n]$ whenever $n\geq 0$ appears in formulas as fixed parameter (which need not be integer).
In the bi-poissonised occupancy scheme the box occupancy counts at time
  $t\geq0$ are
\begin{align}
\Pi_1(t),\ldots,\Pi_{N}(t),
\end{align}
where the number of boxes  is random,
 $N\stackrel{\rm d}{=}{\rm Poisson}(n)$; in particular  with probability $e^{-n}$ the number of boxes  is zero.
For varying $n$ the models are consistent, so that  for  $n'<n$ the model with a smaller number of boxes is obtained from the larger  model by discarding  some number of boxes ${\rm N}(n',n]$.
The number
 of balls allocated by time $t$, equal to $\sum_{i=1}^{N} \Pi_i(t)$, has  a
 compound Poisson distribution with p.g.f.\
$z\mapsto \exp\left(n (e^{t(z-1)}-1)\right)$.

 We denote by $\widehat{M}_{n,i}(t), \widehat{\mu}_{n,r}(t),\widehat{\tau}_n$  the bi-poissonised counterparts of the fixed-$n$ random variables
and set them
equal to zero  in the event   $N$ takes a small value and they are not defined in a natural way.
For instance, in the event $N<\ell+1$ we set
$\widehat{\tau}_n=0$ for the stopping time which, as  before,
 terminates the allocation process as soon as the number of boxes with at most $m$ balls  becomes $\ell$.
Such conventions do not impact the envisaged distributional asymptotics.

The technical benefit of poissonising the number of boxes relates to the marking theorem, which
for many applications can be used
  in the following transparent form. Let ${\rm P}_1,\ldots, {\rm P}_k$    ($1\leq k\leq\infty$) be a collection of mutually exclusive properties
of a countable subset of $[0,\infty]$, such that a unit-rate Poisson point process possesses one of them almost surely.
Let ${\rm N}_j$ be the set of atoms $\theta_i$ whose $\Pi_i$ satisfies ${\rm P}_j$. Then ${\rm N}_1,\ldots, {\rm N}_k$ are independent Poisson processes.

This applied, immediately gives that the bi-poissonised multiplicities  are independent in $r$ (for fixed $n,t$) and
satisfy
\begin{equation}\label{mu-t-P}
\widehat{\mu}_{n,r}(t)\stackrel{\rm d}{=}{\rm Poisson}(n p_r(t)), ~~~r\in{\mathbb Z}_{\geq0}.
\end{equation}

A further consequence concerns independence in the time domain. For $r\geq 1$ define
{\it the $r$th arrival} point process
$\widehat{\rm R}_{n,r}$  to be the set
of $r$th atoms of ${\Pi}_i$'s  for $i\leq N$. That is, instant $t$ accomodates an atom of $\widehat{\rm R}_{n,r}$ if  one of $N$ boxes receives its $r$th ball at time $t$.
(In the occupancy problems these times are
sometimes called $r$-records, see \cite{DGM} and references therein.)

\begin{proposition}\label{records} Each point process $\widehat{\rm R}_{n,r}, r\in {\mathbb Z}_{>0},$ is  Poisson, with the  intensity measure
\begin{align}
{\mathbb E}[\widehat{\rm R}_{n,r}({\rm d}t)]=np_{r-1}(t){\rm d}t, ~t\geq0,
\end{align}
or, equivalently,
\begin{align}
{\mathbb E}[\widehat{\rm R}_{n,r}[t,\infty]]=n {P}_{r-1}(t), ~t\geq0.
\end{align}
\end{proposition}
\begin{proof} For $\theta_i$ an atom of $\rm N$ and  $\eta_i$ the $r$th arrival in $\Pi_i$, we have that the pairs $(\theta_i,\eta_i)$ comprise
a bivariate Poisson process. Projecting these for $i\leq N$ yields $ \widehat{\rm R}_{n,r}$, which is thus Poisson.
\end{proof}
\noindent
The proposition bears  some similarity with
 the $r$-records (rank $r$ observations) in the sense of extreme-value theory \cite[Section 4.6]{Resnick}. The major difference, however, is that  the  $\widehat{\rm R}_{n,r}$'s are not independent for different values of $r$.

The bi-poissonised stopping time  $\widehat{\tau}_n$ identifies as  the $(\ell+1)$st rightmost atom of  $\widehat{\rm R}_{n,m+1}$, hence Proposition \ref{records} yields an explicit  formula for the distribution
\begin{align}
\P[\widehat{\tau}_n\leq t]=P_\ell(nP_m(t)),
\end{align}
analogous to (\ref{taund}).
This can be manipulated to show
 that
$\widehat{\tau}_n-\gaa_n\stackrel{\rm d}{\to}\tau$, where $\tau\stackrel{\rm d}{=}{\rm
  Gumbel}(\ell+1)$
as in the fixed-$n$ model. (Alternatively, condition on $N$ to prove $\ttau_n-\alpha_{N}\dto \tau$
and use $|\alpha_{N}-\gaa_n|=o_p(1)$.)
This becomes most elementary  in the CCP
case $\ell=0, m=0$, where the centering constant is  $\gaa_n=L$ and  the distribution function coincides {\it exactly} with the standard Gumbel distribution on $(-L,\infty)$:
\begin{equation}\label{GumCCP}
{\mathbb P}[\widehat{\tau}_n-L\leq s]=
\exp(-ne^{-s-L})=e^{-e^{-s}},
~~~s\geq -L,
\end{equation}
which entails that the value
$\widehat{\tau}_n=0$
has positive probability $e^{-n}$.

\subsection{Approximating ${\rm M}_n^s$ by $\widehat{\rm M}_n^s$}\label{BinPoAppr}
By the marking theorem the multiplicities $\widehat{\mu}_{n,r}(t), r\in
{\mathbb Z}_{\geq 0},$ for each fixed $t$ are independent, as said above,
which makes
\begin{align}
\widehat{\rm M}_n^s:=\sum_{k=-\infty}^\infty  \widehat{\mu}_{n,b_n+k}(\gaa_n+s)\delta_k
\end{align}
a lattice Poisson point process, with mean multiplicity
\begin{align}\label{lxx2}
np_{b_n+k}(\gaa_n+s)
\end{align}
(equal to zero for $k<-b_n$ or $s\leq-\gaa_n$) at site $k\in{\mathbb Z}$.
We proceed with  estimating the proximity of ${\rm M}_n^s$ to $\widehat{\rm M}_n^s$ restricted  to a region containing the few largest atoms with high probability.

For $b_n$ as in (\ref{bc}) and $r_n<b_n$ yet to be chosen, consider $b_n-r_n$ as a truncation level for the occupancy counts.  The number of boxes $i\leq n$  with
$\Pi_i(\gaa_n+s)>b_n-r_n$ is a binomial random variable $X$ with parameters $n$ and  $\overline{P}_{b_n-r_n}(\gaa_n+s)$. Likewise the number of boxes $i\leq N(n)$ satisfying this condition is a Poisson random variable $Y$
 with the same mean ${\mathbb E}[X]={\mathbb E}[Y]$. A criterion for  selecting the truncation level  is that
the mean number of overshoots goes to $\infty$ but the overshoot probability for  any particular box $i$ approaches $0$.
Choosing $r_n\sim \alpha L$ with some $0<\alpha<e-1$,  Stirling's formula
gives the asymptotics
\begin{align}\label{sb1}
-\log (\overline{P}_{b_n-r_n}(\gaa_n+s))\sim {\beta}L
\end{align}
with some $0<\beta <1$
depending on $\alpha$.
This entails the desired $n \overline{P}_{b_n-r_n}(\gaa_n+s)\to\infty$  and  $\overline{P}_{b_n-r_n}(\gaa_n+s)\to 0$,
to enable application of   Prohorov's Poisson-binomial
bound
(see e.g.\ \cite[p.~2, and also (1.6) and (1.23)]{PoA} and the references there)
that becomes
\begin{align}
d_{\rm TV}(X,Y)=O(n^{-\beta'}),
\end{align}
for any $\beta'<\beta$,
in fact  locally uniformly in $s$.
Appealing to the identity (\ref{dident}) for MBPP's this translates as
\begin{equation}\label{1appr}
d_{\rm TV}({\rm M}_n^s|_{[-r_n,\infty]}\,,\,            \widehat{\rm M}_n^s|_{[-r_n,\infty]})=O(n^{-\beta'})
\end{equation}
(where $r_n$ grows logarithmically as above).
For fixed positive integer $K$,
the events
\begin{eqnarray}
\{X<K \}&=&\{M_{n,K}(\gaa_n+s)\leq b_n-r_n\}= \{M_{n,K}(\gaa_n+s)-b_n\leq -r_n\},\\
\{Y<K\}&=&\{ \widehat{M}_{n,K}(\gaa_n+s)\leq b_n-r_n\}=\{\widehat{M}_{n,K}(\gaa_n+s)-b_n\leq -r_n\}
\end{eqnarray}
have much smaller probabilities,
also as a consequence of  \eqref{sb1},
and by a coupling argument we can force the  $K$ top atoms of ${\rm M}_n^s$
and $\widehat{\rm M}_n^s$ to coincide.
From this, we obtain that, for any fixed $K$,
\begin{align}\label{va1}
\dtv\bigpar{(M_{n,i}(\gaa_n+s))_{i=1}^K,(\widehat{M}_{n,i}(\gaa_n+s))_{i=1}^K}
=O(n^{-\beta'}).
\end{align}
We remark that it is possible to improve on the above by letting $K$ grow as
a small
power of $n$, this way extending the approximation to cover
some intermediate  order statistics.

\subsection{Coupling by the index of box}
An alternative coupling of  a fixed number of
 maximal occupancy counts in the fixed-$n$ and bi-poissonised models
only uses symmetry and not  distribution of the counts.
For fixed $t>0$ and
$n>K$ there exist $K$ distinct boxes whose ordered occupancy counts are $M_{n,1}(t),\ldots, M_{n,K}(t)$.
If $M_{n,K}>M_{n,K+1}$   the set of such boxes is uniquely determined, otherwise we
choose at random a suitable number from the  boxes with $M_{n,K}$ balls.
The labels of thus selected  $K$ boxes is a random sample  from $[n]$; let $I$ be the largest index in the sample.
For $N\stackrel{\rm d}{=}{\rm Poisson}(n)$,
if $N>K$
define $\widehat{I}$ similarly, leaving  the index undefined otherwise.
For $\varepsilon>0$, if the event  $A:=\{I\vee\widehat{I}\leq
n-n\varepsilon<N<n+n\varepsilon\}$ occurs then  we can couple in  such a way that
\begin{equation}\label{PoCo}
(M_{n,i}(t))_{i=1}^K~~~{\rm and}~~~(\widehat{M}_{n,i}(t))_{i=1}^K
\end{equation}
are equal.
Using exchangeability we obtain estimates
(at least if $n\eps$ is an integer)
\begin{align}
  {\mathbb P}[I>n-n\varepsilon]\le K\varepsilon,
\qquad
{\mathbb P}[|N-n|<n\varepsilon ,\widehat{I}>n-n\varepsilon]
\leq \frac{2 K\varepsilon}{1-\varepsilon}.
\end{align}
 Choosing $\varepsilon\sim2\sqrt{L/n}$, these dominate the known tail bound
$
{\mathbb P}[|N-n|\geq n\varepsilon]\leq 2 e^{-n\varepsilon^2/3}
$
(following from Bennet's inequality,
see also  \cite[Corollary 2.3 and Remark 2.6]{JLR});
hence the total variation distance between the vectors  in (\ref{PoCo}) does
not exceed
${\mathbb P}[A^c]<7 K\sqrt{L} n^{-1/2}$ for $n$ not too small.

\subsection{Approximating $\widehat{\rm M}_n^s$ by an exponential process}  By passing to  the Poisson point process $\widehat{\rm M}_n^s$ the approximation problem
is  reduced
to asymptotics of the  intensity measure.
For Poisson random variables $X$ and $Y$,   we have that $d_{\rm
  TV}(X,Y)\leq |{\mathbb  E}[X]-\E[Y]|$.
From this
the total variation distance between lattice Poisson processes does not exceed   the $\ell_1$-distance between their intensity measures, seen as sequences of point masses
at integer locations.
(Stronger bounds exist, as follows from
\cite[Theorem 2.2(i) and its proof]{SJ212}, but we do not need them.)

The intensity measure of ${\rm \Xi}^\uparrow_{b}$ has masses growing exponentially fast in the negative direction. This  forces us to impose a less generous truncation than in Section  \ref{BinPoAppr} to keep the quality of approximation high.

\begin{proposition} For $r_n'\sim \beta \log L$ with any $0<\beta<1/2$,
locally uniformly in $s$,
\begin{equation}\label{2appr}
d_{\rm TV}(\widehat{\rm M}_n^s|_{[-r_n',\infty]}\,,\,  {\rm \Xi}^\uparrow_{(e-1)s+c_n}|_{[-r_n',\infty]})\to\ 0.
\end{equation}
\end{proposition}
\begin{proof} We first restrict both Poisson processes to the interval $(-r_n'\,, L^{1/4})$.  Within this range we approximate the point masses of the intensity measure of  ${\rm \Xi}^\uparrow_{(e-1)s+c_n}$
by their $\widehat{\rm M}_n^s$-counterparts;
recall that these intensities are given by \eqref{lxx1} and \eqref{lxx2}.
To that end, Lemma \ref{LSt} is applied with $u=O(1)$ and $v$ varying in the
range from $-r_n'+O(1)$ to $L^{1/4}+O(1)$.
 A simple calculation, using \eqref{an}, \eqref{bt} and \eqref{bc}, shows that
the pointwise {\it relative} approximation error for the point masses  is then given by the last three terms in (\ref{logP2}), which altogether
are estimated as $cL^{-1/2}\log L$ for some $c>0$, as one easily checks.
Since for the exponential process
\begin{equation}\label{XiTail}
{\mathbb E}\,[{\rm \Xi}^\uparrow_{b}{(-r_n',\infty}]] \asymp e^{r_n'}
\end{equation}
(locally uniformly in $b$), the $\ell_1$-distance between the mean measures on  $(-r_n'\,, L^{1/4})$ is of the order $O(e^{r_n'}L^{-1/2}\log L)$, which is in fact $o(1)$  by our choice of $r_n'$.
By this very token and (\ref{logP1}), for both processes  the total mean
measure of $[L^{1/4},\infty]$ is $O(\exp(-L^{1/4}))$, which makes a negligible contribution to the total variation distance
in (\ref{2appr}).
\end{proof}

Clearly,  (\ref{2appr}) combined with  (\ref{1appr}) yields the desired approximation
\begin{equation}\label{3appr}
d_{\rm TV}({\rm M}_n^s|_{[-r_n',\infty]}\,,\,  {\rm \Xi}^\uparrow_{(e-1)s+c_n}|_{[-r_n',\infty]})\to\ 0.
\end{equation}
Obviously from (\ref{XiTail}), for any fixed $K$, for both processes the $K$ largest atoms exceed $-r_n'$ w.h.p., whence
with the notation of Section \ref{S3} we obtain as a corollary
\begin{theorem}\label{fixedmax}
For every $K\in{\mathbb Z}_{>0}$, as $n\to\infty$ locally uniformly in  $s\in {\mathbb R}$
\begin{equation}\label{ThFM}
d_{\rm TV}\big( (M_{n,i}(\gaa_n+s)-b_n)_{i=1}^K\,,\,(\lceil\xi_i+ (e-1)s+c_n\rceil)_{i=1}^K\big)\to 0,
\end{equation}
where $\xi_i\stackrel{\rm d}{=}{\rm Gumbel}(i)$ is the decreasing sequence of  atoms of\/ $\Xi$.
\end{theorem}

\begin{remark}
(The need for a condition for convergence in distribution.)
The more general asymptotic regime   $t_n\sim \alpha \log n$ was addressed in   \cite[Ch. II, Section 6, Theorem 2]{Kolchin}.
Adjusted to the poissonised model, the cited result claims a weak convergence of $M_{n,1}(t)-r(n,t)$  provided an (integer) centering constant $r=r(n,t)$ is chosen so that $np_r(t)\to\lambda,$
for some $\lambda\in(0,\infty)$.
However, for the generic  sequence $(t_n)$ with such asymptotics the required centering need  not exist,  as exemplified by Lemma \ref{LSt} in the case $\alpha=1$.
Thus in essense the weak convergence necessitates an additional constraint on $(t_n)$ to avoid  oscillations of  certain Poisson probabilities.
\end{remark}

Oscillatory asymptotics for a sequence of distributions  involve closeness to a set of accumulation points of the sequence.
Leaving aside the periodicity patterns, this fits in the  following broad
scenario.
 Let $(S, d_{\rm })$ be a metric space,
and let $\gax\mapsto H_\gax$ be a continuous map from a compact interval
(or, more generally, some compact metric space) $J$ into $S$.

\begin{lemma}\label{LeLem2}
Let  $(G_n)$ be a sequence in $S$ and let $\rho(n)$ be a sequence in $J$.
Then the following are equivalent.
\begin{itemize}
\item[(i)]
For any subsequence $(n_j)$ and $\gax\in J$
such that   $\rho(n_j)\to\gax$ as $j\to\infty$,
we have $G_{n_j}\to H_\gax$.  
\item[(ii)]
As $n\to\infty$,
\begin{equation}\label{dconv2}
d_{\rm }(G_n, H_{\rho(n)})\to0.
\end{equation}
\end{itemize}
Furthermore,
if these hold, then $(G_n)$ is relatively compact.
\end{lemma}
\begin{proof}
(ii)$\Rightarrow$(i).
If (ii) holds and $\rho(n_j)\to\gax$, then
$d(G_{n_j},H_{\rho(n_j)})\to0$ by (ii), and
$d(H_{\rho(n_j)},H_\gax)\to0$ by the assumption on $H$.
Consequently, $d(G_{n_j},H_{\gax})\to0$.

(i)$\Rightarrow$(ii).
Consider any subsequence $(n_j)$.
Since the set $J$ is compact, we may select a subsubsequence $(n'_j)$ such
that $n'_j\to\gax$ for some $\gax\in J$, and thus $G_{n_j'}\to H_\gax$ by the
assumption (i).
Furthermore, $H_{\rho(n_j')}\to H_\gax$ by continuity.
Consequently, as $j\to\infty$,
\begin{align}
  d(G_{n'_j},H_{\rho(n'_j)})
\le
  d(G_{n'_j},H_{\gax})
+
  d(H_\gax,H_{\rho(n'_j)})
\to 0.
\end{align}
Hence every subsequence has a subsubsequence for which \eqref{dconv2} holds;
as is well-known, this implies \eqref{dconv2} for the full sequence.

Finally, if (i) and (ii) hold, then for any subsequence $(n_j)$ we may select a
subsubsequence $(n_j')$ such that $\rho(n_j')$ converges, which by (i)
implies that $(G_{n_j'})$ converges. Hence, $(G_n)$ is relatively compact.
\end{proof}

Applying the lemma to the setting of Theorem \ref{fixedmax},  we take for
$(S,d)$ the space of probability distributions on $\bbZ^K$ with $d=d_{\rm TV}$,
$J:=[0,1]$,
and  $\rho(n):=c_n$.
Then \eqref{ThFM} is an instance of \eqref{dconv2}, and (i) in Lemma
\ref{LeLem2} describes subsequential limits in distribution,
for subsequences with $c_n\to \gax$ for some $\gax\in\oi$.

\subsection{Stopped maxima}  \label{S5.5}
We have now all prerequisites  to derive
one of our main results,
a multivariate extension  of
Theorems \ref{TI} and \ref{TI2},
from Fact \ref{FACT}.
The latter says that
conditionally on $\tau_n=t$, the
$n-\ell-1$ box occupancy counts exceeding $m$ at this time,
excluding the box that gets the ball at $\tau_n$,
are i.i.d.\ with the truncated Poisson  distribution
\begin{align}\label{trunc}
  {\mathbb P}[\Pi_i(t)=r\mid\Pi_i(t)\geq m+1]
=\frac{p_r(t)}{\overline{P}_{m}(t)},
\qquad r\geq m+1.
\end{align}
Recall also the shorthand notation  $M_{n,i}=M_{n,i}(\tau_n)$ for the box occupancies when the
allocation is stopped.

\begin{theorem}\label{stoppedmax}
For every $K\in{\mathbb Z}_{>0}$, as $n\to\infty$,
\begin{equation}\label{ThSM}
d_{\rm TV}\left((M_{n,i}-b_n)_{i=1}^K\,,\,(\lceil\xi_i+ (e-1)\tau+c_n\rceil)_{i=1}^K\right)\to 0,
\end{equation}
where $\tau\stackrel{\rm d}{=}{\rm Gumbel}(\ell+1)$ is independent of\/ $\Xi$.
\end{theorem}
\begin{proof}
For $t=\gaa_n+O(1)$ we have $P_{m}(t)=O(n^{-1})$, therefore the  asymptotics in
Lemma \ref{LSt} hold for the truncated Poisson distribution in \eqref{trunc}
as well;
we may also replace $n$ by $n-\ell-1$ without changing the result.
In view of this, it follows from the fact mentioned before the theorem that
the approximation in Theorem \ref{fixedmax} remains valid also  for
$(M_{n,i}(\tau_n)-b_n)_{i=1}^K$ conditioned on $\tau_n-\gaa_n=s$, locally
uniformly in $s$.
In other words,
conditionally on $\tau_n$ and with $\Xi=\set{\xi_i}$ independent of $\tau_n$,
\begin{equation}\label{thsm3}
d_{\rm TV}\left((M_{n,i}(\tau_n)-b_n)_{i=1}^K\,,
\,(\lceil\xi_i+ (e-1)(\tau_n-\gaa_n)+c_n\rceil)_{i=1}^K\right)\to 0,
\end{equation}
uniformly for $\tau_n-\gaa_n$ in a compact set.
Recall that $\tau_n-\gaa_n\stackrel{\rm d}{\to}\tau$.
In particular, $\tau_n - \gaa_n$ is tight,
and thus it follows that \eqref{thsm3} holds also unconditionally.

If $c_{n_j}\to c_0$  for some subsequence $(n_j)$ then
furthermore, along the subsequence,
\begin{align}\label{thsm4}
\left(\lceil\xi_i+ (e-1)(\tau_n-\gaa_n)+c_n\rceil\right)_{i=1}^K
\stackrel{\rm d}{\to}
\left(\lceil\xi_i+ (e-1)\tau+c_0)\rceil\right)_{i=1}^K
\end{align}
by the mapping theorem \cite[Theorem 5.1]{Billingsley}
(since $x\mapsto\ceil{x}$ a.s.\ is continuous at $\xi_i+ (e-1)\tau+c_0$);
since the random variables in \eqref{thsm4} take values in the countable set
$\bbZ^K$, it follows
by Scheff{\'e}'s lemma 
that \eqref{thsm4} holds also in total variation.
Hence, \eqref{thsm3} implies that, along $(n_j)$,
\begin{equation}\label{thsm5}
d_{\rm TV}\left((M_{n,i}(\tau_n)-b_n)_{i=1}^K\,,
\,(\lceil\xi_i+ (e-1)\tau+c_0)\rceil)_{i=1}^K\right)\to 0.
\end{equation}
Finally, Lemma \ref{LeLem2} enables us to pass
from the convergence of subsequences to the claimed approximation \eqref{ThSM}.
\end{proof}

\subsection{Equivalent formulations}

By analogy  with \refT{TI3}, we may formulate \refT{stoppedmax} in an equivalent way
for the non-centered variables $M_{n,i}=M_{n,i}(\tau_n)$.
\begin{theorem}\label{stoppedmax3}
For every $K\in{\mathbb Z}_{>0}$, as $n\to\infty$,
\begin{equation}\label{ThSM3}
d_{\rm TV}\left((M_{n,i})_{i=1}^K\,,\,(\lceil\xi_i+ (e-1)\tau+\tb_n\rceil)_{i=1}^K\right)\to 0,
\end{equation}
where $\tau\stackrel{\rm d}{=}{\rm Gumbel}(\ell+1)$ is independent of\/
$\Xi=\{\xi_i\}$.
\end{theorem}

As emphasised in \refS{SSvague}, it is technically convenient to regard the point
processes in focus as point processes on $\oooox$, although they never have an
atom at $\infty$.
Let $\cN\oooox$ be the space of locally finite integer-valued measures on
$\oooox$, and
regard a point process on $\oooox$
as a random element of $\cN\oooox$.
The space $\cN\oooox$ is equipped with the vague topology, which
is metrisable (and Polish) \cite[Proposition 3.1]{BP}.
Using this
framework we can state \refT{stoppedmax} in the following equivalent
form that involves  a limit of the entire set
$\set{\Pi_i(\tau_n)}_1^n=\set{M_{n,i}}_1^n$ regarded as a point process.
\begin{theorem}\label{TPP}
  Let $d$ be any metric on the space $\cN\oooox$ that induces the vague
  topology.
Then
\begin{equation}\label{tpp}
d\bigpar{\set{M_{n,i}-b_n}_{i=1}^n\,,\,
\set{\lceil\xi_i+ (e-1)\tau+c_n)\rceil}_{i=1}^\infty}\to 0.
\end{equation}
\end{theorem}
\begin{proof}
Consider a subsequence $(n_j)$ such that $c_{n_j}\to c_0$ for some $c_0\in\oi$.
Then, by the virtue of
 \refL{L:point1},  we obtain from
\eqref{thsm5} that, along the subsequence,
\begin{equation}\label{tpp2}
d\left(\set{M_{n,i}-b_n}_{i=1}^n\,,
\,\bigset{\lceil\xi_i+ (e-1)\tau+c_0\rceil}_{i=1}^\infty\right)\to 0.
\end{equation}
The result now follows from \refL{LeLem2}.
\end{proof}

At last, instead of looking at few rightmost atoms, we may restrict our
point processes to a vicinity of $\infty$. This leads by the virtue of Lemma
\ref{L:point1} (and \refL{LeLem2})
to the following equivalent version.
\begin{theorem}\label{TPPs} For every $r\in{\mathbb Z}$
\begin{equation}\label{tppS}
d_{\rm TV} \bigpar{{\rm M}_n^{\tau_n-\alpha_n}|_{[r,\infty]}, \Xi^\uparrow_{(e-1)\tau+c_n}|_{[r,\infty]}}\to 0.
\end{equation}
\end{theorem}

\section{Proof of the main result by interpolation}\label{Spf}
\noindent
\refT{TPP}
shows that the lattice point process of stopped counts $\set{\Pi_i(\tau_n)}_{i=1}^n$
may be approximated by the exponential Poisson process $\Xi$ that is shifted and then has
all atoms rounded to integers.
We present here an alternative proof based on the idea of interpolation of the lattice process to ${\mathbb R}$, that amounts to
artificially
adding the `missing' fractional parts to the atoms and then showing convergence to an exponential
Poisson process on $\bbR$. By this approach, the oscillations are revealed only at
the final stages of the argument.

The shift operation on $\Xi$  from Section \ref{S3.1} makes sense for  arbitrary point process on ${\mathbb R}$.
Thus,
for a point process ${\rm H}=\set{\eta_i}$ and a real number $b$, we let ${\rm H}\pm b$
denote the shifted processes
\begin{align}\label{ema}
 {\rm  H}+b:=\set{\eta_i+b},
\qquad   {\rm H}-b:=\set{\eta_i-b}
\end{align}
obtained by translating each atom the same way.

For convenience of notation, let
$W_t
\stackrel{\rm d}{=}
\Po(t)$, and let $W_t'$ and $W_t''$ have the
truncated
 (conditioned)
distributions
$W_t'
\stackrel{\rm d}{=}
(W_t\mid W_t\ge m+1)$ and $W_t''\eqd(W_t\mid W_t\le m)$.
By Fact \ref{FACT}, if we condition on $\tau_n=t$,
then  the occupancy numbers
$\set{\Pi_i(\tau_n)}_{i=1}^n$ are given by
\begin{align}\label{pip}
\set{\Pi'_i(t)}_{i=1}^{n-\ell-1}
\cup
\set{\Pi''_j(t)}_{j=1}^{\ell}
\cup\set{m+1},
\end{align}
where all random variables are independent and have the distributions
$\Pi'_i(t)\eqd W_t'$ and
$\Pi''_j(t)
\stackrel{\rm d}{=}
W_t''$.
The last $\ell+1$ of the numbers in \eqref{pip} are $\le m+1$ and may be
ignored, as will be seen below,
so only the $n-\ell-1$ numbers $\Pi_i'(t)$ are important asymptotically.

Let $E\stackrel{\rm d}{=}
\Exp(1)$ and consider
$\rand \stackrel{\rm d}{=}(E\mid E<1)$;
 %
thus $\rand$
is a random variable in $[0,1)$ with  the distribution function
\begin{align}\label{ee4}
\P[R\le x]= \frac{1-e^{-x}}{1-e\qw},
\qquad 0\le x\le 1.
\end{align}
Let  $\rand_i$ ($i\ge 1$)
be independent copies of $\rand$,
also independent of all other variables.
We define, for a given $t=t_n$,  the modified variables
\begin{align}\label{e5}
  \tPii_i:=\Pi'_i(t_n)+\rand_i
\end{align}
and note that we are back to integer counts via
$\Pi'_i(t_n)=\floor{\tPii_i}$.

Let $s_n$ and $x_n$ be any bounded sequences of real numbers, and
consider only $n$ that are so large that $\log n+s_n\ge0$.
Let, recalling \eqref{an}, \eqref{bt}, and $L:=\log n$,
\begin{align}
\label{tn}
t_n&:=\gaa_n+s_n,
\\\label{yn}
y_n&:=\tb_n+x_n,
\\\label{kn}
k_n&:=\floor{y_n},
\\\label{x'n}
x'_n&:=k_n-\tb_n
.\end{align}
Note that
\begin{align}\label{e66}
  x_n-x'_n=y_n-k_n =\frax{y_n}\in[0,1).
\end{align}
Then, for $i\le n-\ell-1$ we have, with $W'_{t_n}$ as above and independent of
$\rand_i$, using \eqref{ee4},
\begin{align}\label{e7}
  \P[\tPii_i>y_n\mid \taun = t_n]
&=\P[W'_{t_n}+\rand_i>y_n]
\\\notag&
=\P[W'_{t_n}>k_n]
+
\P[
W'_{t_n}=k_n]
\P[\rand_i>y_n-k_n]
\\\notag&
=\P[W'_{t_n}>k_n]
+
\P[
W'_{t_n}=k_n]
\frac{e^{-(y_n-k_n)}-e\qw}{1-e\qw}.\end{align}
We have $t_n\to\infty$, and thus
$\P
[W_{t_n}\le m]
\to0$.
Hence, by Lemma \ref{LSt} and simple calculations,
\begin{align}\label{e9}
  \P[
W'_{t_n}=k_n]
&
=  \P[W_{t_n}=k_n\mid W_{t_n}>0]
=   \P[W_{t_n}=k_n]
\etto
\\\notag&
=\frac{e-1}{n} e^{(e-1)s_n-x_n'+o(1)}.
\end{align}
and similarly
\begin{align}\label{e8}
  \P[W'_{t_n}>k_n]&
=   \P[W_{t_n}>k_n]\etto
=\frac{1}{n} e^{(e-1)s_n-x_n'+o(1)}.
\end{align}
Then \eqref{e7}--\eqref{e8} yield, recalling  \eqref{yn} and \eqref{e66},
\begin{align}\label{ea1}
  \P[\tPii_i>\tb_n+x_n]&
=\frac{1}{n} e^{(e-1)s_n-x_n'+o(1)}
+
\frac{e-1}{n} e^{(e-1)s_n-x_n'+o(1)}
\frac{e^{-(y_n-k_n)}-e\qw}{1-e\qw}
\\\notag&
=\frac{1}{n} e^{(e-1)s_n-x_n'}\Bigpar{e^{o(1)}+e^{o(1)}\bigpar{e^{1-(x_n-x'_n)}-1}}
\\\notag&
=\frac{1}{n}\Bigpar{ e^{(e-1)s_n+1-x_n}+o(1)}.
\end{align}

Let $x\in\bbR$ and choose  $x_n:=x+(e-1)s_n+1$.
Furthermore, define
\begin{align}\label{ea3}
  \hPi_i(s_n):=\tPii_i-\tb_n-(e-1)s_n-1
=\Pi'_i(t_n)+R_i-\tb_n-(e-1)s_n-1
.\end{align}
Then \eqref{ea1} yields
\begin{align}\label{ea2}
\P[\hPi_i(s_n)>x]&
=\P[\tPii_i>\tb_n+x_n]
=\frac{1}{n}\Bigpar{ e^{-x}+o(1)}.
\end{align}
Note that  $\hPi_i(s_n)$ depends on the chosen bounded
sequence $s_n$,
both directly and through $t_n$,
but  the \rhs{} of \eqref{ea2} does not.

The random variables $\hPi_i(s_n)$
are independent for $1\le i\le n-\ell-1$.
Consequently, if we define the point process
\begin{align}\label{ea5}
  \hXiisn&:=\bigset{\hPi_i(s_n)}_{i=1}^{n-\ell-1}
\\\notag&\phantom:
= \bigset{\Pi'_i(t_n)+R_i}_{i=1}^{n-\ell-1}  -\tb_n-(e-1)s_n-1
,\end{align}
then \eqref{ea2} shows by the standard Poisson convergence of binomial
distributions that, still for
any bounded sequence $s_n$,
\begin{align}\label{ea6}
{\hXiisn(x,\infty]}&
\dto\Po\bigpar{ e^{-x}}.
\end{align}
Since
$\Xi(x,\infty]\stackrel{\rm d}{=}\Po\bigpar{e^{-x}}$,
cf.\ \eqref{EXi},
we thus obtain from \eqref{ea6}
\begin{align}\label{ea7}
\hXiisn(A)&
\dto\Xi(A)
\end{align}
for every interval $A=(x,\infty]$.

This is not quite enough to show convergence in distribution in
the space $\cN\oooox$, but it is not far from it.
Let $\cU$ be the family of all finite unions $\bigcup_1^k(u_j,v_j]$ with
$-\infty<u_j<v_j\le\infty$.
For any such set $A\in\cU$, we can use \eqref{ea2} for $x=u_j$ and $v_j$,
$j=1,\dots,k$, and conclude that
\begin{align}
  \P[\hPi_i(s_n)\in A]&
=\frac{1}{n}\bigpar{\mu(A)+o(1)},
\end{align}
where $\ddx\mu(x)=e^{-x}\dd x$ is the intensity measure of $\Xi$,
and it follows as above that
\eqref{ea7} holds for every $A\in\cU$.
This implies convergence
\begin{align}\label{ea8}
\hXiisn&
\dto\Xi
\end{align}
in $\cN\oooox$, see for example
\cite[Theorem 4.15]{Kallenberg-rm}.
(In the terminology there, $\cU$ is a dissecting ring, and we may take
$\cI=\cU$; both conditions in the theorem follow from \eqref{ea7} for
$A\in\cU$.
See also the version in \cite[Proposition 16.17]{Kallenberg}.)
Alternatively, \eqref{ea8} follows easily from \eqref{ea2} using
\cite[Corollary 4.25]{Kallenberg-rm}; we leave the details to the reader.

Define $S_n:=\tau_n-\gaa_n$, and recall from \eqref{taulim} that
\begin{align}\label{ea88}
  S_n\dto \tau,\quad{\rm with~~}~~~\tau\stackrel{\rm d}{=} \textrm{Gumbel}(\ell+1)
.\end{align}
Define also the point processes
\begin{align}\label{em0}
  \tXi_n&:=\bigset{\Pi_i(\tau_n)+R_i}_{i=1}^n,
\\\label{em00}
  \hXisn&:=\tXi_n-(e-1)S_n-\tb_n-1
=\bigset{\Pi_i(\tau_n)+R_i-(e-1)S_n-\tb_n-1}_{i=1}^n.
\end{align}
Then, by \eqref{tn} and \eqref{pip},
\begin{align}\label{em1}
  \bigpar{\tXi_n\mid S_n=s_n}
=  \bigpar{\tXi_n\mid \tau_n=t_n}
\eqd
\set{\Pi'_i(t_n)+R_i}_{i=1}^{n-\ell-1}
\cup
\set{\Pi''_j(t_n)+R_{n-j}}_{j=0}^{\ell},
\end{align}
where we for convenience let $\Pi''_n(t_n):=m+1$.
(We have only equality in distribution in \eqref{em1}, since the equality
involves a harmless relabelling of $R_1,\dots,R_n$.)
Hence, also using \eqref{em00} and \eqref{ea5},
\begin{align}\label{em2}
  \bigpar{\hXisn\mid S_n=s_n}
&
=  \bigpar{\tXi_n\mid S_n=s_n}-(e-1)s_n-\tb_n-1
\\\notag&
\eqd
\hXiisn
\cup
\bigset{\Pi''_j(t_n)+R_{n-j}-(e-1)s_n-\tb_n-1}_{j=0}^{\ell},
\end{align}
For $0\le j\le\ell$, we have $\Pi''_j(t_n)\le m+1$ and thus,
recalling $s_n=O(1)$,
\begin{align}\label{em3}
  \Pi''_j(t_n)+R_{n-j}-(e-1)s_n-\tb_n-1 = -\tb_n+O(1) \to -\infty.
\end{align}
Hence, if $A\in\cU$ is as above, then for large $n$, the final multiset in
\eqref{em2} is disjoint from $A$, and thus
\eqref{em2} shows that \eqref{ea7} holds also for
$\bigpar{\hXisn\mid S_n=s_n}$,
which, as for \eqref{ea8} above, implies
\begin{align}\label{em4}
\bigpar{\hXisn\mid S_n=s_n}
\dto \Xi
\end{align}
in $\cN \oooox$.
(Alternatively, this follows from \eqref{ea8}, \eqref{em2}, and \eqref{em3}
using \refL{L:point1}.)

We have shown \eqref{em4}
for any bounded sequence $s_n$.
Thus \refL{L2} below applies and yields
\begin{align}\label{eb1}
  \bigpar{\hXisn, S_n}\dto (\Xi,\tau),
\end{align}
with $\Xi$ and $\tau\stackrel{\rm d}{=}\Gumbel(\ell+1)$ independent.
Consequently, by \eqref{em00} and the continuous mapping theorem,
\begin{align}\label{eb2}
  \bigset{\Pi_i(\tau_n)+R_i-\tb_n-1}_1^n&
=\hXisn+(e-1)S_n
\dto \Xi+(e-1)\tau.
\end{align}
This is our continuous version of Theorems \ref{stoppedmax}--\ref{TPP},
where we have added artificial fractional parts $R_i$ in order to get a nice
limit  $\Xi+(e-1)\tau$
consisting of the Poisson process $\Xi$ with an independent random shift
$(e-1)\tau$.

To obtain
the desired conclusions about the occupancy counts
 it now remains only to remove the fractional parts.
Arrange $\tPix_i:=\Pi_i(\tau)+R_i$ in decreasing order as
$\tPix\sss1\ge\tPix\sss2\ge\dots$, and note that then
\begin{align}\label{em5}
M_{n,i}=\floor{\tPix_{(i)}}.
\end{align}
\refL{L:point1} shows that \eqref{eb2} is equivalent to,
with $\Xi=\set{\xi_i}_{i=1}^\infty$ as in \refSS{S3.1},
\begin{align}\label{eb3}
  \bigpar{\tPix\sss{i}-\tb_n-1}_{i=1}^K
\dto
\bigpar{\xi_i+(e-1)\tau}_{i=1}^K
\end{align}
for every fixed $K\ge1$.

We write as in \eqref{bc} $\tb_n=b_n+c_n$
where $b_n:=\floor{\tb_n}$ is an integer and
$c_n:=\frax{\tb_n}\in[0,1)$
is the fractional part.
Consider a subsequence such that $c_n\to \gax$ for some $\gax\in\oi$.
Then, along this subsequence, it follows from \eqref{eb3} that
\begin{align}\label{eb4}
  \bigpar{\tPi\sss{i}-b_n}_{i=1}^K
\dto
\bigpar{\xi_i+(e-1)\tau+1+\gax}_{i=1}^K.
\end{align}
The $K$ variables on the \rhs{} of \eqref{eb4} have continuous
distributions, and are thus a.s. not integers;
hence, the vector on the \rhs{} is a.s.\ a continuity point of
the mapping $F:\bbR^K\to\bbR^K$ given by
$(z_i)_1^K\mapsto(\floor{z_i})_1^K$.
Consequently, by
\cite[Theorem 5.1]{Billingsley}, we may apply this mapping $F$ and conclude,
using \eqref{em5} and \eqref{eb4}, that
\begin{align}\label{eb5}
  \bigpar{M_{n,i}-b_n}_{i=1}^K
&=
  \bigpar{\floor{\tPi\sss{i}-b_n}}_{i=1}^K
\\&\notag
\dto
\bigpar{\floor{\xi_i+(e-1)\tau+1+\gax}}_{i=1}^K
=
\bigpar{\ceil{\xi_i+(e-1)\tau+\gax}}_{i=1}^K.
\end{align}
This yields \refT{stoppedmax} by \refL{LeLem2},
again taking $d=\dtv$ and letting $G_n$ and $H_\gax$ be distributions
of the random vectors on left and right sides of \eqref{eb5}.

\refTs{stoppedmax3} and \ref{TPP} follow easily as in \refSS{S5.5}.

\subsection{A general lemma}
We used in the proof above the following simple lemma on joint convergence
using conditional distributions.  We admit this may belong to the folklore,
and give a detailed proof since we
are not aware of any explicit reference.

\begin{lemma}\label{L2}
Let $(X_n,Y_n)$, $n\ge1$, be a sequence of pairs of random variables taking
  values   in $\cX\times\cY$ for some Polish spaces $\cX$ and $\cY$.
Let $X$ be a random variable in $\cX$ and
suppose that there exists  regular conditional distributions
$\cL(X_n\mid Y_n=y)$, $y\in\cY$, such that, as \ntoo,
for any convergent sequence
$y_n\to y$ in $\cY$,
\begin{align}\label{lx1}
  (X_n\mid Y_n=y_n) \dto X
.\end{align}
Suppose further that $Y_n\dto Y$ as \ntoo, for some random
variable $Y$ in $\cY$.
Assume, as we may, that $X$ and $Y$ are independent.
Then, as \ntoo,
\begin{align}\label{lx2}
  (X_n,Y_n)\dto(X,Y).
\end{align}
\end{lemma}
\noindent
The assumption \eqref{lx1} means that $\cL(X_n\mid Y_n=y_n)$ converges to the
distribution $\cL(X)$. Note that the limit distribution does not depend on
$y$.
\begin{proof}
Let $A\subseteq\cX$ and $B\subseteq\cY$ be such that
$\P[X\in\partial A]=0$ and $\P[Y\in\partial B]=0$
(i.e., these are continuity sets for  $X$ and $Y$).
We have
\begin{align}\label{lx3}
  \P[(X_n,Y_n)\in A \times B]&
=\E\left[
\P[X_n\in A\mid Y_n]\indic{Y_n\in B}\right],
\end{align}
where we use the regular conditional distributions in the assumption.
By the Skorohod coupling theorem \cite[Theorem~4.30]{Kallenberg},
we may assume that $Y_n\to Y$ almost surely.
Then \eqref{lx1} implies
$\P[X_n\in A\mid Y_n]\to \P[X\in A]$ a.s.\ (since $A$ is an $X$-continuity set),
and $Y_n\to Y$ implies $\indic{Y_n\in B}\to\indic{Y\in B}$ a.s.\
(since $B$ is a $Y$-continuity set).
Consequently, \eqref{lx3} implies by the dominated
convergence theorem that
\begin{align}\label{lx4}
  \P[(X_n,Y_n)\in A \times B]&
\to\E\left[\P(X\in A)\indic{Y\in B}\right]
=\P[X\in A]\P[Y\in B]
\\\notag&
=\P[(X,Y)\in A\times B].
\end{align}
This implies \eqref{lx2} by \cite[Theorem 3.1]{Billingsley}.
\end{proof}

\section{Dynamical aspects}
\label{Sdyn}
\subsection{Small counts}
The number of empty  boxes decreases each time a box receives its first ball.
Functional limit theorems for this process in the setting of the discrete-time occupancy scheme were first obtained  by Sevastyanov
 \cite{Sev} through asymptotic analysis of the multivariate p.g.f.\
 of  the finite-dimensional distributions.
For the regime of interest here, Theorem 5 of the cited paper showed  (in a minor disguise) the  convergence to an exponential  Poisson process (see also \cite[Ch.~4 Section 5]{Kolchin}).
In the context of CCP an equivalent result was
proved quite recently
by another method  for the process of first arrivals   (see \cite[Theorem 4.3.38]{Mladen} and references therein),
although the connection with \cite{Sev} was apparently overlooked.
Ilienko \cite[Theorem 3.1]{Ilienko} used the poissonised scheme
to identify the Poisson limits for the processes of $r$-th arrivals.

We aim next to  demonstrate, in the framework of   the bi-poissonised occupancy scheme, how the time evolution of  small counts  connects to
the processes of $r$th arrivals.

Recall from  Proposition \ref{records} that  the pre-limit processes $\widehat{\rm R}_{n,r}$  of $r$-arrivals ($r\geq1$)
are nonhomogeneous Poisson point processes
with intensity measure $np_{r-1}(t){\rm d}t, ~t\geq0$.
Let
\begin{align}
\alpha_{n,r}:=L+(r-1)\log L-\log(r-1)!
\end{align}
which is  (\ref{an}) with $m=r-1$.
Employing  (\ref{ingamma}) and  (\ref{mSt}) to control the intensity measure,
we obtain  for every fixed $r\geq 1$ convergence of Poisson point processes
in the form of their counting functions
\begin{equation}\label{Rlim}
(\widehat{\rm R}_{n,r}(\alpha_{n,r} +s,\infty], ~s\in {\mathbb R})\stackrel{\rm d}{\to}(\Xi(s,\infty],~~s\in {\mathbb R}).
\end{equation}
In fact, it is easy to see that on every fixed interval $(a,\infty]$,
the intensities converge in $L^1$, and thus the intensity measures in total
variation, and \eqref{Rlim} follows.
Taking the de-poissonisation of the numbers of balls and boxes for granted,  (\ref{Rlim})  recovers the cited
results from \cite{Ilienko, Mladen}.

We stress that the $r$-arrival processes have no common asymptotic time scale,
meaning that  $\widehat{\rm R}_{n,r'}(\alpha_{n,r} +s,\infty]$
for $r'< r$ converges in probability to $0$, and  for $r'> r$ converges in probability to $\infty$.
These relations are just  features of the right $(r-1)$-domain in terms of  \cite{Kolchin}.

Sevastyanov's result on convergence of the process of empty boxes
is  equivalent to the $r=1$ instance of (\ref{Rlim}) by the  virtue of   identity
\begin{align}
N-\widehat{\mu}_{n,0}(t)=\widehat{\rm R}_{n,1}(t),
\end{align}
which holds pathwise a.s.\  for all $t\geq0$,
and just says that the number of empty boxes decreases each time some  box receives its first ball.
In this formula
$N\stackrel{\rm d}{=}{\rm Poisson}(n), ~  \widehat{\mu}_{n,0}(t) \stackrel{\rm d}{=} {\rm Poisson}(ne^{-t})$, with
$N-\widehat{\mu}_{n,0}(t)$ and $\widehat{\mu}_{n,0}(t)$ being independent for each fixed $t$.
The number of empty boxes  $(\widehat{\mu}_{n,0}(t), ~t\geq 0)$ is  a
pure-death process with unit  death rate per capita, and  $(\widehat{\rm
  R}_{n,1}(t), ~t\geq0)$ is a Poisson process with the exponential jump rate $ne^{-t}$.

To generalise for all $r\geq1$, we observe that
\begin{equation}\label{sumc}
\widehat{\mu}_{n,r-1}(t) +\sum_{k=0}^{r-2}\widehat{\mu}_{n,k}(t)=\widehat{\rm R}_{n,r}(t,\infty],~t\geq0.
\end{equation}
The  term $\widehat{\mu}_{n,r-1}(t)$ is placed separately to emphasise that  the remaining sum  is asymptotically negligible for $t=\alpha_{n,r}+O(1)$
(which corresponds to the right $(r-1)$-domain of \cite{Kolchin}).
We center and use \eqref{Rlim},
which, using that the separated term in (\ref{sumc}) dominates the rest,
yields
(again in the Skorohod space $D(-\infty,\infty]$):
\begin{proposition} For $r\in{\mathbb Z}_{>0}$,  as $n\to\infty$,
\begin{align}
(\widehat{\mu}_{n,r-1}(\alpha_{n,r}+s),~s\in{\mathbb R}) \stackrel{\rm d}{\to}  (\Xi(s,\infty],~~s\in {\mathbb R}).
\end{align}
\end{proposition}
\noindent
Notably,  although  $(\widehat{\mu}_{n,r}(t),~t\geq0)$ for $r>0$ is a nonmonotonic birth-death process with upward jumps occuring at rate $np_{r-1}(t)$ and downward at  rate one per capita,
in  the limit  there are only downward jumps  at times of a nonhomogeneous Poisson process.
This feature of  the occupancy scheme in  the right $(r-1)$-domain is new to our knowledge.

\subsection{Parallels with queueing theory}
For a more comprehensive  picture of the dynamics,  the  bi-poissonised model should be considered as a whole. Representing the occupancy counts as a  random measure  on the integer lattice,  in the form
\begin{align}
\sum_{r=0}^\infty  \widehat{\mu}_{n,r}(t)\delta_r
=\sum_{i=1}^{N}\delta_{\widehat{M}_{n,i}},
\end{align}
and letting the time $t\geq0$ vary defines a Markov process with values in the space of point measures $\cN \oooox$,  with the random starting state  $N\delta_0$ (compare with $n \delta_0$ for the fixed-$n$ poissonised scheme)
 and
 transitions driven by  a Poisson flow
 \begin{align}
\sum_{r=0}^\infty  \widehat{\mu}_{n,r}(t+h) \delta_r \stackrel{\rm d}{=}T_{h}
\circ\sum_{r=0}^\infty  \widehat{\mu}_{n,r}(t) \delta_r.
\end{align}

In   intuitive terms, we may  think of a box as a particle  jumping at unit probability rate, independently of other particles,  from one site on the lattice to the next on its right, with  each such  move caused by a new ball added to the box.
Site $0$ has a distinguished role as a source  emitting particles  that proceed
to enter  site $1$ at epochs of
the  Poisson process $\widehat{\rm R}_{n,1}$  with nonhomogeneous rate $ne^{-t}$.
 The configuration on ${\mathbb Z}_{>0}$ (representing the multiplicities of nonzero box counts)
 evolves like an infinite network of infinite-server queues
 connected in series   \cite{MW} (also known as  a tandem of  ${\rm M}_t\slash{\rm }{\rm M}\slash \infty$ queues).
In this context, the facts like independence of multiplicities in (\ref{mu-t-P}) and features of  the arrival processes in Proposition \ref{records} appear as specialisation of  properties of  open networks with  Poisson inputs
\cite{Kelly}.
Other way round, our large $n$ results on the occupancy problem for times $t=\alpha_{n,r}+O(1)$
(as well as results for other temporal regimes \cite{Kolchin})
admit transparent  interpretation in terms of  a heavy-traffic  approximation for the  series of infinite-server queues
  with exponential  input rate $ne^{-t}$.

A similar representation  has been used to study a continuous-time  growth process of random permutations \cite{GS}.
In that model, like in other combinatorial processes
related  to tandem networks, the focus is on convergence to stationarity
of the configuration of particles on  any fixed finite  set of sites, see  \cite{tandem} for examples.

\section{Asymptotic independence of extreme occupancy counts}

\subsection{Asymptotic independence}
The bi-poissonised multiplicities $\widehat{\mu}_{n,r}(t)$'s
for different $r$  are only independent for  every fixed $t$ (and $n$) but  not as processes: e.g.\ the larger the number of boxes with $r$ balls, the higher is the likelihood that the number of boxes with $r+1$ balls will increase in the nearest time.
Nevertheless,   Theorem \ref{stoppedmax} suggests that the   independence of
 $\Xi$ and  $\tau$ stems from  some kind of  weak dependence  of   a vector of  $K$ maximal and a vector of $K$ minimal occupancy counts for times $t=\alpha_n+O(1)$.
With the latter shown, we will  be in position
to   approximate  the joint distribution of the measure-valued process
$(\widehat{\rm M}_n^s, ~s\in{\mathbb R})$ taken together with  $\widehat{\tau}_n-\gaa_n$, thus putting
 (\ref{ThSM})  (through the bi-poissonised version of the result) in light of the theorem on continuity of compositions \cite{Whitt}.

To introduce the appropriate independence concept formally, consider a bivariate sequence of random elements $(X_n, Y_n)$ taking values in a product Polish space
$\cX \times\cY$. We say that $X_n$ and $Y_n$ are {\it asymptotically independent} if
for  independent  $X_n'$ and $Y'_n$ with $X_n'\stackrel{\rm d}{=}X_n$, $Y_n'\stackrel{\rm d}{=}Y_n$ it holds that
\begin{align}
\lim_{n\to\infty}d_{\rm TV}((X_n, Y_n), (X_n', Y_n'))=0.
\end{align}
(Then every subsequential weak limit will be a product measure.)
See \cite[Condition AI-4]{DN}) for this and weaker forms of  asymptotic independence.

We are interested in times around the instant $L=\log n$.
The covering interval with endpoints
\begin{equation}\label{trt}
t_0:=\frac{2}{3}L, \qquad
t_1:=L+\left(K-1\right)\log L
\end{equation}
will serve our purpose. The processes in the next lemma are to be considered as random elements
of the space of  cadlag functions endowed with the Skorohod topology.

\begin{lemma} For  $K>1$
and $t_0, t_1$ given by \eqref{trt}
the $K$-variate extreme-value processes
\begin{equation}\label{KminKmax}
((\widehat{M}_{n,n-i+1}(t))_{i=1}^K,~ t\in[t_0,t_1]) ~~~{\rm and~~~}   ((\widehat{M}_{n,i}(t))_{i=1}^K,~ t\in[t_0,t_1])
\end{equation}
are asymptotically independent.
\end{lemma}
\begin{proof}
For the time being
let us regard boxes with at most $K$ balls as `small'
 and the others as `big'.
Accordingly, we
 split the sequence of  multiplicities
into two blocks
\begin{align}
S(t)=(\widehat{\mu}_{n,r}(t))_{r=0}^K,
\qquad B(t)=(\widehat{\mu}_{n,r}(t))_{r=K+1}^\infty,
\end{align}
which  for every fixed $t$ are independent,
and stem from two complementary collections of boxes.

The  block $S(t)$ of small box multiplicities is a Markov process, whose
 lifetime
until absorption at zero is $L+K\log L+\Op(1)$,  in consequence of
 the discussion around   the centering constant  (\ref{an}) (now with $K$ assuming the role of  $m$).
The lifetime exceeds $t_1$ by  $\log L+\Op(1)$, hence at  time $t_1$  w.h.p.\ the number of small boxes is at least $K$ and, by monotonicity, the $K$ minimal occupancy counts for all $t\leq t_1$ are due to small
boxes. So we are reduced  to show that the process of small box multiplicities on $[t_0,t_1]$
is asymptotically independent of the $K$ maximal box occupancy counts.

To that end, for times $t\geq t_0$
we further decompose the process of big boxes as
\begin{align}
B(t)=B'(t)+B''(t),
\end{align}
where $B'(t)$ is the sequence of multiplicities representing occupancy counts of  those big boxes
that contained more than $K$ balls already at time $t_0$, and $B''(t)$  appears due to the increase of the content of small boxes.
Independence of the blocks at $t_0$ and the Poisson flow dynamics entail that the processes
 $(S(t),\, t\geq t_0)$ and $(B'(t),\, t\geq t_0)$ are independent.
It remains to show that
for the  range $t\in[t_0, t_1]$, the nonzero multiplicities in $B(t)$
that account for the $K$ maximal box occupancy counts
coincide w.h.p.\ with their counterparts in  $B'(t)$.
That is to say, we assert that
 boxes small at time $t_0$ are unlikely to overtake the largest ones at later stages up to time $t_1$.

Indeed, by Theorem \ref{fixedmax} (with $m=0$) and \eqref{va1},
for every $\varepsilon>0$ w.h.p.
\begin{align}
  \label{va2}
(e-\varepsilon)L <\widehat{M}_{n,K}(L)
\leq \widehat{M}_{n,1}(L)<(e+\varepsilon)L,
\end{align}
 where we recall $L=\log n$.
Hence, by monotonicity  for $t\geq t_0$, also
\begin{equation}
\label{tcom}
\frac{2}{3}(e-\varepsilon)L < \widehat{M}_{n^{2/3},K}(t_0)
\le\widehat{M}_{n,K}(t_0)
\le\widehat{M}_{n,K}(t).
\end{equation}
On the other hand, the total number
$\sum_{r=0}^K\widehat{\mu}_{n,r}(t_0)$
of
 small boxes existing at time $t_0$  has a Poisson distribution with mean
\begin{align}
 nP_K(t_0)\leq ne^{-t_0}t_0^K<n^{1/3} L^K.
\end{align}
Hence it satisfies
$\sum_{r=0}^K\widehat{\mu}_{n,r}(t_0)<n^{1/2}$ w.h.p., which implies that
  the maximum number of balls any of these boxes can contain
at  a   later  time $t$
(i.e., the index of the largest nonzero component of $B''(t)$) for
$t_0\leq  t\leq t_1<  t_0+L/2$   does not exceed $K+J$, where
$J\stackrel{\rm d}{=}\widehat{M}_{n^{1/2},1}(L/2)$.
By \eqref{va2}
\begin{align}
\widehat{M}_{n^{1/2},1}(L/2)
<(e+\varepsilon)\frac{L}{2}< \frac{2}{3}(e-\varepsilon)L
\end{align}
w.h.p.\  for $\varepsilon<e/7$. Comparing with  (\ref{tcom}) yields
$K+J<\widehat{M}_{n,K}(t_0)$ w.h.p., which
shows the claim above that w.h.p.\
for $t\in[t_0,t_1]$ the $K$ largest box occupancy counts are not represented.
by $B''(t)$.
Thus,  in (\ref{KminKmax})  the $K$-variate minimal process coincides
w.h.p.\ with the $K$ minimal counts contributing to
$S$,  and the maximal process coincides w.h.p.\ with the $K$ maximal counts contributing to $B'$, where $S$ and $B'$ are independent.
The proof is completed by appealing to \cite[Proposition 3]{DN} which ensures the asserted asymptotic independence.
\end{proof}

\begin{corollary}\label{C6.2}
$\htau_n {\rm ~ and~} ((\widehat{M}_{n,i}(t))_{i=1}^K,~ t\geq 0)$
are asymptotically independent for each  $K\in {\mathbb Z}_{>0}$.
\end{corollary}
\begin{proof} We may assume $K\geq m+2$ to have  w.h.p.\ $t_0<\widehat{\tau}_n<t_1$,
for the bounds defined in (\ref{trt}).
The  truncated   stopping time $\tau'_n= (\widehat{\tau}_{n}\vee t_0)   \wedge t_1$
is adapted to the minimal process in  (\ref{KminKmax}), hence asymptotically independent
of the maximal process. Since $\tau'_n= \widehat{\tau}_{n}$ w.h.p. we can apply \cite[Proposition 3]{DN} again.
\end{proof}

\subsection{Stopped maxima via  continuity of compositions}
The value of a random process at a random time is sometimes referred to as composition.
The continuity of compositions theorems  connect convergence of such evaluations
with the convergence of underlying processes and times.
We sketch the ingredients needed for an alternative proof of Theorem \ref{stoppedmax}
following  this thread.

Firstly, we have observed a weak convergence of ${\tau}_n-\gaa_n$ to some  random variable $\tau$.
As a next step,  with a minor extra effort   Theorem \ref{fixedmax} extends
to a functional approximation result in the sense of Lemma \ref{LeLem2}. In particular,
for $n$ running along a subsequence of integers  with $c_n\to c_0\in [0,1]$, the functional convergence
\begin{equation}\label{fcon}
({\rm M}_{n}^s, s\in {\mathbb R})\stackrel{\rm d}{\to} (\Xi^\uparrow_{(e-1)s+c_0}, s\in{\mathbb R}),
\end{equation}
follows from the marginal convergence for each fixed $s=s_0$ and the fact that both processes are driven by the same Poisson flow.
Convergence  (\ref{fcon})
and the asymptotic independence in Corollary \ref{C6.2} allow one to control the joint distribution to show
that
\begin{equation}\label{jconv}
\left(\tau_n-\gaa_n, ({\rm M}_n^s, ~s\in {\mathbb R})\right)\stackrel{\rm d}{\to}
\left(\tau, (\Xi^\uparrow_{(e-1)s+c_0},~ s\in{\mathbb R})\right),
\end{equation}
where $\tau$ and the limit process are independent.
The composition theorem from
\cite[Corollary 13.3.2, p.~433]{Whitt} now applies to
yield convergence of  the stopped  point process
\begin{align}
{\rm M}_n^{\tau_n-\gaa_n} \stackrel{\rm d}{\to}
\Xi^\uparrow_{(e-1)\tau+c_0}
\end{align}
along the subsequence.
Finally,  the full extent of Theorem \ref{stoppedmax} with oscillatory asymptotics obtains by Lemma \ref{LeLem2}.

A version of the following formula for the joint distribution of stopped
maximum and its multiplicity was stated in \cite[Theorem 12]{Ivch1}  without proof.
\begin{corollary}\label{joint-mm} For fixed $k\in {\mathbb Z}$, $j\in{\mathbb Z}_{>0}$ and $M_n:=M_{n,1}(\tau_n)$,
\begin{align}
&{\mathbb P}[M_{n}-b_n=k,~\mu_{n, b_n+k}(\tau_n)=j]=\\\notag&\qquad
\int_{-\infty}^\infty p_0\left(e^{(e-1)s+ c_n-k}\right)p_j\left((e-1)e^{(e-1)s+c_n-k}\right)  e^{-s}p_\ell(e^{-s}){\rm d}s+o(1).
\end{align}
\end{corollary}
\begin{proof}
We need to compute the analogous probability for the stopped approximating process $\Xi^\uparrow_{(e-1)\tau+c_n}$.
Recalling the intensity measure (\ref{lxx1})--(\ref{lxx1s}) we obtain
 ${\mathbb P}[\Xi^\uparrow_{b}[k+1,\infty]=0]=p_0(e^{b-k})$ and    ${\mathbb P}[\Xi^\uparrow_{b}(\{k\})=j]=p_j((e-1)e^{b-k})$,
for the events which determine a $j$-fold rightmost atom at location $k$.
Conditionally on $\tau=s$, we multiply these probabilities while setting
 $b=(e-1)s+c_n$, then   integrate in $s$  over  the Gumbel$(\ell+1)$ density of $\tau$
given in \eqref{sa3}.
\end{proof}

\section{Exponential tail estimates and moments}\label{Smom}
\noindent
We proceed with uniform in (large)
$n$ exponential tail estimates for the maximal stopped occupancy counts in the fixed-$n$ poissonised scheme.
Apparently the underlying  light-tail phenomenon has not been given due attention in the literature.
We take therefore first  a wider view on maximal order statistics, complementing the established theory found  in \cite{Mikosch, Reiss}.

\subsection{The general setting}\label{GenS}
Consider  a sequence of distribution functions $F_n$  on ${\mathbb R}_+$,  and let  $X_{n,1}\geq \cdots \geq X_{n,n}$
be an  ordered i.i.d.\  sample from    $F_n$.  Suppose  it is possible to choose   an approximate upper $1/n$ quantile, i.e.\ to find $x_n$ satisfying
\begin{equation}\label{bsf}
c_-n^{-1}\leq\overline{F}_n(x_n)\leq c_+ n^{-1},
\end{equation}
where  $\overline{F}_n:=1-F_n$ and  $c_-, c_+$ are some positive constants. If $F_n$ is continuous then, of course,
a $1/n$  quantile can be chosen exactly. Assuming that
\begin{equation}\label{Tail-F}
\frac{\overline{F}_n(x+y)} {\overline{F}_n(x)}\leq C e^{-cy}  {\rm~~~\quad  for~\quad~} \theta x_n\leq x\leq x_n~\quad~~ {\rm and ~\quad~~}y\geq0,
\end{equation}
with some positive constants $c, C$, and $\theta\in(0,1)$,
we wish to conclude on  a similar tail estimate for  the centered  statistic $X_{n,i}-x_n$ with fixed $i$.

\begin{lemma} \label{Le-Tails}
Under the assumptions {\rm (\ref{bsf}) } and {\rm (\ref{Tail-F})}, for $y\geq0$
and fixed $i\ge1$,
\begin{eqnarray}
  {\mathbb P}[X_{n,i}-x_n>y]
&\leq&
 (c_+ C) e^{-c y},
\label{upTail}
\\
\label{doTail}
{\mathbb P}[X_{n,i}-x_n\leq -y]
& \leq& C_2 e^{-c_2e^{c(1-\theta) y}},
\end{eqnarray}
with some constant $C_2>0$ and $c_2=c_-/(2C)$.
\end{lemma}

\begin{proof}
For the right tail estimate we only need the upper bound in (\ref{bsf})
and that  (\ref{Tail-F}) holds with $x=x_n$. Granted that, we have for $y>0$
\begin{align}
  {\mathbb P}[X_{n,i}-x_n>y]&\leq  {\mathbb P}[X_{n,1}-x_n>y]=1- F_n^n(x_n+y)
\leq
n\overline{ F}_n(x_n+y)
\\\notag&
\leq n \overline{F}_n (x_n) C e^{-c y}  \leq  (c_+ C) e^{-c y}.
\end{align}

The left tail requires more effort. For $0\leq  y\leq x_n$ write the exact formula
\begin{equation}\label{bs1}
{\mathbb P}[X_{n,i}-x_n\leq - y]=\sum_{j=0}^{i-1} {n\choose  j}\overline{F}_n^j(x_n-y)F_n^{n-j}(x_n-y).
\end{equation}
To bound this sum from the above we recall that the binomial distribution is stochastically increasing
as  the success probability increases; therefore it is enough to estimate $\overline{F}_n(x_n-y)$ from below. Inverting
(\ref{Tail-F}),
we find
\begin{equation}\label{Tail-FI}
\frac{\overline{F}_n(x_n-y)} {\overline{F_n}(x_n)} \geq    C^{-1} e^{cy},
\qquad 0\leq y\leq (1-\theta)x_n.
\end{equation}
Hence,
 noting for $0\leq y\leq (1-\theta)x_n$ that $C^{-1} e^{cy}   \overline{F}_n(x_n)\le 1$ by \eqref{Tail-FI},
we obtain for
the binomial sum in (\ref{bs1})  an upper bound
\begin{align}\label{va3}
\sum_{j=0}^{i-1} \binom{n}{j} & \big(C^{-1} e^{cy}   \overline{F}_n(x_n)\big)^j        \big(1-C^{-1} e^{cy}   \overline{F}_n(x_n)\big)^{n-j}
\\\notag&
\leq
\sum_{j=0}^{i-1}\frac{1}{C^jj!} \big(  n\overline{F}_n(x_n)\big)^j    e^{cjy}
\exp\left(-(n-j)\overline{F}_n(x_n) C^{-1} e^{cy}\right)
\\\notag&
\leq
\sum_{j=0}^{i-1}\frac{c_1^je^j}{ j!}  \exp\left(c jy  - 2c_2e^{cy}\right)
\leq
C_1  \exp\left(c iy  - 2c_2e^{cy}\right)
\leq C_2 e^{-c_2 e^{cy}}
\end{align}
where $c_1=c_+/C$, $c_2=c_-/(2C)$, and $C_1,C_2$ are some constants;
the final inequality holds since $\sup (az-e^{cz})<\infty$
for every $a,c>0$.
This implies \eqref{doTail} for $0\le y\le(1-\theta)x_n$.
In particular, the bound \eqref{va3} is valid for the cutoff $y=(1-\theta)x_n$. Therefore,
for the remaining range $(1-\theta)x_n\leq y\leq x_n$
we have
\begin{equation}\label{bs}
{\mathbb P}[X_{n,i}-x_n\leq -y]\leq {\mathbb P}[X_{n,i}-x_n\leq -(1-\theta)x_n]\leq  C_2 e^{-c_2e^{c(1-\theta) x_n}} \leq C_2 e^{-c_2e^{c(1-\theta) y}}.
\end{equation}
Consequently, combining \eqref{va3} and \eqref{bs},
we obtain (\ref{doTail}).
\end{proof}
The striking asymmetry between the right and left tails is partly explained by  a similar behaviour of the maximal point in
the exponential process $\Xi$,
 which has the Gumbel distribution \eqref{gumbel}.
Even so the above estimates do not presume approximability or
 convergence of $X_{n,i}$'s in distribution.
Replacing the double exponent in (\ref{doTail}) by a weaker exponential bound,  we have
\begin{equation}\label{2sidedB}
{\mathbb P}[|X_{n,i}-x_n|>y]\leq C_0 e^{-c y}
\end{equation}
with suitable $C_0>0$. This two-sided estimate will be sufficient for our purposes,
but see Lemma \ref{LT} below.

\subsection{Gamma and Poisson examples}\label{GP} We illustrate the obtained tail bounds for maximal order statistics
 in two examples relevant to our stopped occupancy problem.

If $F$ is Gamma$(m+1, 1)$ with $m\geq 0$, then the hazard rate
$h(x)=F'(x)/\bF(x)$
is increasing to $1$. From this, (\ref{Tail-F}) holds in the form
\begin{equation}\label{upTg}
\frac{\overline{F}(x+y)} {\overline{F}(x)}=\exp\left( -\int_x^{x+y} h(z){\rm d}z\right) \leq e^{-h(x)y}\leq e^{-(1-\varepsilon)y},~~~~~\hskip0.5cm \varepsilon>0,
\end{equation}
for large enough $x$.

For another  example, suppose $F_n$ is {\rm Poisson}$(t_n)$ with $t_n\sim L=\log n$. The Poisson distribution also has an increasing hazard
rate (as being log-concave); thus for $x_n\sim eL$, and hence $t_n/x_n\to e^{-1}$,  (\ref{pPas}) gives
\begin{equation}\label{upTp}
\frac{\overline{P}_{x_n+k}(t_n)}{\overline{P}_{x_n}(t_n)}\le \left(\frac{\overline{P}_{x_n+1}(t_n)}{\overline{P}_{x_n}(t_n)}\right)^k
\leq e^{-(1-\varepsilon)k}, ~~~~~\hskip0.5cm k>0,
\end{equation}
for large enough $n$,
similarly to the Gamma example above.

\subsection{Tail estimates for stopped maximal  occupancy counts}
We combine the Gamma and Poisson bounds to obtain
tail estimates for $M_{n,i}(\tau_n)$.  The idea comes from the property of
the Poisson distribution in Lemma \ref{LSt},
 which  tells us that  for $t\sim L$ an increment $u$ of the parameter is compensated by about $v=(e-1)u$ change of the quantile.

Throughout the subsection, $k$ is a nonnegative integer;
$C_1,C_2,\dots$ and $c_1,c_2,\dots$ are strictly positive constants
that may disagree with those in Section \ref{GenS}.
\CCreset\ccreset
\begin{lemma}  \label{LT}
For any fixed $i\ge1$ and $\eps>0$,
there exist constants $C_{\cdot}$ and $c_{\cdot}$ (that may depend on $i$ and $\eps$)
such that for all $n\ge i$ and $k\ge0$,
\begin{align}\label{lt1}
{\mathbb P}[M_{n,i}(\tau_n)-b_n>k]&
\leq  \CC e^{-(1-\varepsilon)k/e},
\\\label{lt2}
{\mathbb P}[M_{n,i}(\tau_n)-b_n\le-k]&
\leq  \CC e^{-\cc e^{\ccname{\ccb} k}}
\leq  \CC e^{- k}
.\end{align}
\end{lemma}

\begin{proof}
The estimates \eqref{lt1}--\eqref{lt2} are more or less trivial for
  each fixed $n$, so we may assume that $n$ is large when needed.

We consider first $\tau_n$.
Recall the realisation of  $\tau_n$ as a maximal order statistic from
Gamma$(m+1,1)$, and note that $\ga_n$ is an approximate upper $1/n$
quantile, see \eqref{an}--\eqref{mSt}.
Taking $x_n=\ga_n$, we see from \eqref{upTg} that
(at least for large $n$)
\eqref{Tail-F} holds with
$c=1-\eps$ and $\theta=\eps$.
Hence, \refL{Le-Tails} yields
\begin{align}\label{tau2>}
  {\mathbb P}[\tau_n-\gaa_n> k/e]&\leq  \CC e^{-(1-\varepsilon)k/e},
\\\label{tau2<}
  {\mathbb P}[\tau_n-\gaa_n< -k/e]&\leq  \CC e^{-\cc e^{(1-2\eps) k/e}}.
\end{align}

Consider now $M_{n,i}(\tau_n)$.
For the right tail, the event $M_{n,i}(\tau_n)-b_n>k$ implies that either
$\tau_n-\gaa_n>k/e$, or  $\tau_n-\gaa_n\leq k/e$ and then  $M_{n,i}(\gaa_n+k/e)-b_n>k$. Thus with the account of $M_{n,i}\leq M_{n,1}$,
\begin{equation}\label{stMup}
{\mathbb P}[M_{n,i}(\tau_n)-b_n>k]\leq  {\mathbb P}[\tau_n-\gaa_n> k/e]+
{\mathbb P}[M_{n,1}(\gaa_n+k/e)-b_n>k]
.\end{equation}
For the first part on the \rhs,  we apply (\ref{tau2>}).
 For the second part,  we have a
 bound
\begin{align}
{\mathbb P}[M_{n,1}(\gaa_n+k/e)-b_n>k] & =
1-(P_{b_n+k}(\alpha_n+k/e))^n\leq n  \overline{P}_{b_n+k}(\gaa_n+k/e)
\\&\notag
\leq C_2 e^{-k/e},
\end{align}
where the last inequality follows from
\refL{LSt} with $t=\ga_n+k/e$ and $r=b_n+k$ and thus $u=k/e+O(1)$ and
$v=k+O(1)$,
by discarding some negligible or negative terms in
\eqref{logP1}--\eqref{logP2}.
This proves \eqref{lt1}.

For the left tail, the event $M_{n,i}(\tau_n)-b_n\le-k$ implies that either
$\tau_n-\gaa_n<-k/e$, or  otherwise $\tau_n-\gaa_n\geq - k/e$ and then
$M_{n,i}(\gaa_n-k/e)-b_n \leq M_{n,i}(\tau_n)-b_n\le-k$.
Splitting this way  yields
\begin{equation}\label{stMdo}
{\mathbb P}[M_{n,i}(\tau_n)-b_n\le-k]\leq  {\mathbb P}[\tau_n-\gaa_n<- k/e]+
{\mathbb P}[M_{n,i}(\gaa_n-k/e)-b_n\le-k]
,\end{equation}
where the first part is estimated using  (\ref{tau2<}).
To bound  the  left tail of  $M_{n,i}(\gaa_n-k/e)$, we only need to take
care of $k$ within the range $k\leq b_n$,
since otherwise $\P[M_{n,i}(\gaa_n-k/e)-b_n\le-k]=0$.
We consider first $k\le (1-\eps)b_n$
using \refL{LSt} with
$t=\gaa_n-k/e$ and $r=b_n-k$, and thus $u=-k/e+O(1)$ and $v=-k+O(1)$;
note that in this range
\begin{align}\label{gbg1}
  \frac{u}{L} \ge -\frac{(1-\eps)b_n}{eL}+o(1) \sim -(1-\eps)
\end{align}
and thus \eqref{sec-a1} holds.
The right-hand side of \eqref{logP2} becomes $-L+k/e+O(1)$,
and thus \eqref{logP1} yields
\begin{align}\label{gbg2}
\bP_{b_n-k}(\gaa_n-k/e)\ge \ccname{\cca} n\qw e^{k/e}.
\end{align}
The event $M_{n,i}(\gaa_n-k/e)-b_n\le-k$ holds when less than $i$ of the
occupancy counts $\Pi_j(\gaa_n-k/e), j\in[n],$ are greater than $b_n-k$.
We may thus argue as in \eqref{bs1} and \eqref{va3}
(with $\overline{F}_n(x_n-y)$ replaced by $\overline{P}_{b_n-k}(\gaa_n-k/e)$)
and obtain from \eqref{gbg2}
\begin{align}\label{gbg3}
  {\mathbb P}[M_{n,i}(\gaa_n-k/e)-b_n\le-k]
\le \CC e^{ik/e-\cca e^{k/e}}.
\end{align}
The first inequality in \eqref{lt2} follows from \eqref{stMdo},
\eqref{tau2<}, and \eqref{gbg3}, under our assumption $k_n\le(1-\eps)b_n$.
The remaining range $(1-\eps)b_n\leq k\leq b_n$ is dealt with as in
(\ref{bs}).

Finally, the second inequality in \eqref{lt2} is trivial.
\end{proof}
The proof shows (replacing $\eps$ by $\eps/3$) that
we may take $\ccb$ as $(1-\eps)/e$.

\subsection{Moments of the stopped maximal  occupancy counts: approximability}

Complementing Theorem \ref{stoppedmax}, we assert that the analogous result  also holds for the mean and higher moments, in the natural sense.
\begin{theorem}\label{stMom}
 For fixed $i, k\in{\mathbb Z}_{>0}$,
as \ntoo,
\begin{align}\label{stmom}
  {\mathbb E}[(M_{n,i}(\tau_n)-b_n)^k]= {\mathbb E}[   \lceil \xi_i+(e-1)\tau+c_n\rceil^k ]+o(1).
\end{align}
\end{theorem}
\begin{proof}
The exponential tail bounds in \refL{LT}
imply that the sequence $(M_{n,i}(\tau_n)-b_n)^k$ is uniformly integrable.
If  the fractional parts $c_{n_j}$ converge to some $c_0$ along a subsequence $(n_j)$, then  Theorem \ref{stoppedmax} ensures
a weak convergence
\begin{align}
M_{n_j,i}(\tau_{n_j})-b_{n_j}\stackrel{\rm d}{\to}\lceil\xi_i+(e-1)\tau+c_0\rceil,
\end{align}
which together with the uniform integrability imply the convergence of all moments along $(n_j)$.
The assertion now readily follows from the fact that every infinite set of
positive integers contains such $(n_j)$ (see \refL{LeLem2}).
\end{proof}

\subsection{Computing the asymptotic moments}
Extending notation (\ref{dist-Z}), denote the approximating variable
\begin{equation}\label{Zni}
Z_{n,j}:=\lceil \xi_j+(e-1)\tau+c_n \rceil,
\end{equation}
where $\xi_j\stackrel{\rm d}{=}{\rm Gumbel}(j)$ and $\tau\stackrel{\rm
  d}{=}{\rm Gumbel}(\ell+1)$ are independent. These have the familiar
characteristic functions, which are easily shown from e.g.\ the density
\eqref{sa3},
\begin{equation}\label{cf}
{\mathbb E}[e^{\ii x\xi_j}]=\frac{\Gamma(j-\ii x)}{(j-1)!},
\qquad
{\mathbb E}[e^{\ii x\tau}]=\frac{\Gamma(\ell+1 -\ii x)}{\ell!}
\end{equation}
and expected values
\begin{equation}\label{m-xi-tau}
{\mathbb E}[\xi_j]=\gamma-H_{j-1}, \qquad {\mathbb E}[\tau]=\gamma-H_{\ell},
\end{equation}
where $\gamma\doteq0.57721$ is the Euler constant and $H_k:=\sum_{j=1}^k 1/j$ (so $H_0:=0$).

To evaluate the mean  of  (\ref{Zni}) we may apply \cite[Theorem
2.3]{SJ175}, which asserts that for a continuous
random variable $X$ with characteristic function $\varphi$
\begin{equation}\label{round-m}
{\mathbb E}\,
\lceil X \rceil=\E(X)+\frac{1}{2}+\sum_{k\in {\mathbb Z}\setminus\{0\}} \frac{\varphi (2\pi k) }{2\pi \ii k},
\end{equation}
provided $\varphi(x)=O(|x|^{-\varepsilon})$ for $x\to\pm\infty$.  For (\ref{Zni}) this condition is readily justified using the functional
recursion
\begin{equation}\label{gamma-r}
\Gamma(k-\ii x)={(-\ii x)_k}\,\Gamma(-\ii x),
\end{equation}
where $(x)_\ell$ denotes the Pochhammer factorial,
together with
the reflection formula
\cite[5.5.1 and 5.5.3]{NIST},
which yield (see also \cite[5.11.9]{NIST}),
\begin{equation}\label{gamma-inf}
|\Gamma(\ii x)|^2 =\frac{\pi}{x \sinh \pi x}\sim \frac{2\pi}{|x|}e^{-\pi|x|}, ~~~\qquad~x\to\pm\infty.
\end{equation}
Applying (\ref{round-m}) and (\ref{cf})--(\ref{m-xi-tau}), for $i\geq 1$ and $\ell\geq 0$,
\begin{align}\label{mean-Z}
{\mathbb E} [Z_{n,j}]&= \gamma e-H_{j-1}-(e-1)H_\ell + \frac{1}{2}+c_n
\\\notag&\qquad
+\sum_{k\in {\mathbb Z}\setminus\{0\}}
 \frac{ \Gamma(j -2\pi \ii   k)   \Gamma(\ell+1 - 2\pi(e-1)\ii  k)}{(j-1)!\,\ell!}\cdot \frac{e^{2\pi \ii k c_n} }{2\pi \ii k }
.\end{align}
Consequently, \eqref{stmom} yields, recalling \eqref{bc},
\begin{align}\label{mean-M}
  \E [M_{n,j}] &= b_n+\E Z_{n,j}+o(1)
\\\notag&
= a_n +\gamma e-H_{j-1}-(e-1)H_\ell + \frac{1}{2}
\\\notag&\qquad
+\sum_{k\in {\mathbb Z}\setminus\{0\}}
 \frac{ \Gamma(j -2\pi \ii   k)   \Gamma(\ell+1 - 2\pi(e-1)\ii
  k)}{(j-1)!\,\ell!}\cdot \frac{e^{2\pi \ii k c_n} }{2\pi \ii k }
+o(1).
\end{align}
The sum in \eqref{mean-Z} and \eqref{mean-M}
is a Fourier series with  small and rapidly
decreasing coefficients, as is seen from (\ref{gamma-r})--(\ref{gamma-inf}).

For example, in the case $j=1,\ell=0$ (CCP and dixie cup problems) the coefficients have asymptotics
\begin{align}\label{wi1}
(2 \pi k)\qw\bigabs{\gG(1-(e-1) 2 \pi \ii k)\gG(1- 2 \pi \ii k)}
&=
\frac{\pi(e-1)^{1/2}}{(\sinh(2\pi^2k)\sinh(2(e-1)\pi^2k))^{1/2}}
\\&\notag
\sim {2\pi(e-1)^{1/2}}e^{-e\pi^2k}.
\end{align}
The terms with $k=\pm1$ are the largest  by far and have the absolute value
about $0.18379 \cdot 10^{-10}$;  hence the sum in \eqref{mean-Z} makes to the mean a small,
oscillating on the  $\log n$ scale
contribution, whose amplitude does not exceed  $4\cdot 10^{-11}$.

The Fourier coefficients increase with $j$ and $\ell$, though remain
small. For instance, for $\ell=0$ the coefficient for $k=\pm1$  is bounded
in absolute value, for all $j\geq 1$, by (since $|\E[e^{2\pi\ii k\xi_j}]|\le1$)
\begin{align}
  (e-1)\bigabs{\gG(-\ii(e-1)2\pi)} \doteq 0.56552 \cdot 10^{-7}.
\end{align}

Replacing $\tau$ in (\ref{Zni}) by a constant gives random variables approximating (unstopped) maximal order statistics from a Poisson distribution,
as in Theorem \ref{ThFM}. In that setting the first Gamma factor in the sum (\ref{mean-Z}) disappears, making the fluctuations  somewhat larger.
The intensity measure of  our approximating process $\Xi^\uparrow$ has masses decreasing geometrically, so naturally
the  maxima in the occupancy scheme behave similarly to the maxima in samples from a geometric distribution
 (see for the latter \cite{KiPr} and \cite[Example 4.3]{SJ175}).
Making this comparison, it should be noted that
in the occupancy regime of interest here  the parameter of Poisson
distribution (the mean number of balls) changes together with $n$ (the
number of boxes),
but the parameter of the
asymptotic geometric distribution is a fixed value $1-e^{-1}$ that does not
depend on $n$ (see \eqref{logP1}).

Similar Fourier series with small coefficients
(typically involving a Gamma function)
are also known from many different problems, see e.g.\
the examples in
\cite{HP} and
\cite[Sections 2 and 4]{SJ175} and the references there.
In the present situation, the terms in the sum in \eqref{mean-M}
contain a
product of two Gamma functions,
which makes the coefficients
even smaller than in many other similar examples.

The technique  from  \cite{SJ175} may be further applied to obtain formulas for  the variance and higher moments of $Z_{n,j}$.

\section{Multiplicity of the maximum}\label{M-max}

\noindent
Finally we consider the multiplicity of the stopped maximum occupancy count,
\begin{align}
Q_n:=\min \{i: M_{n,i}>M_{n,i+1}\}=\mu_{n,r^*}(\tau_n),
\end{align}
where $r^*=\max\{r: \mu_{n,r}(\tau_n)>0\}$.
The distribution of $Q_n$ does not converge,
 because of oscillations,
but we can  obtain a good approximation by turning to its counterpart
for a randomly shifted  exponential process.

For a shift parameter $u\in{\mathbb R}$ let $\chi_j(u)$ be the probability that the rightmost atom of   $\Xi^\uparrow_{(e-1)\tau+u}$  has multiplicity $j$.
\begin{theorem}\label{TQ}
As \ntoo, for every fixed $j\in{\mathbb Z}_{>0}$,
\begin{equation}\label{tqs}
\P[Q_n=j]
=\chi_j(a_n)+o(1)
=\chi_j(c_n)+o(1),
\end{equation}
where $a_n$ and $c_n=\{a_n\}$ are given by \eqref{bt}--\eqref{bc}, and
 $\qx_j$ introduced above is  a continuous, $1$-periodic function,   representable  by the
Fourier series
\begin{align}\label{tq2s}
  \qx_j(u) =
\frac{\bigpar{1-e^{-1}}^j}{j}
\left(1+
\sum_{k\in{\mathbb Z}\setminus\{0\}}
\frac{\gG({j-2\pi \ii k })
\gG(\ell+1-2\pi (e-1)k\ii)}{(j-1)!\ell!}
e^{2\pi \ii k u}\right).
\end{align}
\end{theorem}
\begin{proof}
The approximability (\ref{tqs}) follows straight
 by Theorem \ref{ThSM} (or Corollary \ref{joint-mm}).
The continuity and   $1$-periodicity of $\qx_j$ both follow from
\eqref{xi+} and
the exponential intensity (\ref{lxx1}), since
for  $c\in {\mathbb Z}$ the shift
\begin{align}
\Xi_{b+c}^\uparrow\stackrel{\rm d}{=}\Xi_b^\uparrow+c
\end{align}
preserves the multiplicity of the maximum; in particular,
$\qx_j(a_n)=\qx_j(c_n)$.

To prove (\ref{tq2s})
we start with evaluating a simpler probability, denoted $q_j(u)$,  of the event that the rightmost atom of $\Xi_{u}^\uparrow$ has multiplicity $j\geq 1$. Arguing as in Corollary \ref{joint-mm},
and manipulating the Poisson probabilities,
\begin{align}
q_j(u)&=\sum_{k\in{\mathbb Z}} p_j((e-1)e^{u-k})p_0(e^{u-k})
\\\notag&
=(1-e^{-1})^j  \sum_{k\in{\mathbb Z}} p_0((e-1)e^{u-k})p_0(e^{u-k})
  \frac{e^{(u-k+1)j} }{j!}
\\\notag&=
(1-e^{-1})^j     \sum_{k\in{\mathbb Z}} p_0(e^{u-k+1})    \frac{e^{(u-k+1)j} }{j!}
\\\notag&=
(1-e^{-1})^j     \sum_{k\in{\mathbb Z}} p_j(e^{u-k+1})
\\\notag&=
(1-e^{-1})^j     \sum_{k\in{\mathbb Z}} p_j(e^{u-k}).
\end{align}
We cannot evaluate this sum explicitly, but it is easy to find its Fourier transform:
\begin{align}\label{se7s}
  \hq_j(k)&:=\intoi e^{-2\pi\ii k u} q_j(u)\dd u
\\\notag&\phantom:
=\frac{\bigpar{1-e^{-1}}^j}{j!}\sum_{r\in{\mathbb Z}} \intoi
 e^{j(u-r)}e^{-e^{u-r}}
e^{-2\pi\ii k u}\dd u
\\\notag&\phantom:
=\frac{\bigpar{1-e^{-1}}^j}{j!}\intoooo
 e^{j u}e^{-e^{u}} e^{-2\pi\ii k u}\dd u
\\\notag&\phantom:
=\frac{\bigpar{1-e^{-1}}^j}{\ell!}\intoo
 e^{-v} v^{j-2\pi\ii k-1}\dd v
\\\notag&\phantom:
=\frac{\bigpar{1-e^{-1}}^j}{j!}\gG({j-2\pi\ii k}).
\end{align}
The Fourier coefficient (\ref{se7s}) appeared  in \cite{Brands} without proof
(in Lemma 4.3 of that paper
the value $\lambda=1$ corresponds to the case of sampling from
Geometric$(1-e^{-1})$), and was also identified in \cite{KiPr} by the calculus of residues.

We return to   $\qx_j$  and note that, for $u\in\bbR$,
\begin{align}\label{se9}
  \qx_j(u)=\E[ q_j\xpar{(e-1)\tau+u}].
\end{align}
We calculate the Fourier coefficients again: for $k\in{\mathbb Z}$,
\begin{align}\label{se10}
  \hqx_j(k)&
=\intoi e^{-2\pi\ii ku}\qx_j(u)\dd u
=\E \intoi e^{-2\pi\ii ku}q_j\xpar{(e-1)\tau+u}\dd u
\\\notag&
=\E \intoi e^{-2\pi\ii k(v-(e-1)\tau)}q_j(v)\dd v
= \hq_j(k)\cdot \E [ e^{2\pi \ii k(e-1)\tau}]
\\\notag&
=
\frac{\bigpar{1-e^{-1}}^j}{j!}
\gG({j-2\pi k\ii })
\frac{\gG(\ell+1-2\pi (e-1)k\ii)}{\ell!},
\end{align}
where by the change of variable we used $1$-periodicity,  and for the last
step we used (\ref{se7s}) and the characteristic function (\ref{cf}) of
$\tau\stackrel{\rm d}{=}{\rm Gumbel}(\ell+1)$.
This completes the proof of \eqref{tq2s}.

The last step of the proof can also be interpreted as follows:
The formula \eqref{se9} implies that $\chi_j$ is the
convolution of $q_j$ and the density function of $-(e-1)\tau$,
and \eqref{se10} then is the standard fact that the Fourier transform of the
convolution of two functions is the product of their Fourier transforms.
\end{proof}

  The Fourier series in \eqref{tq2s} is similar to \eqref{mean-Z}.
As there, the Fourier
coefficients in \eqref{tq2s} decrease rapidly as $|k|$ increases, and the sum
is a very small oscillating term.
For example, for $j=1$ and $\ell=0$, we have already for $k=\pm1$,
similarly to \eqref{wi1},
\begin{align}\label{tq3}
\bigabs{\hqx_1(\pm1)}
&=(1-e\qw)
\bigabs{
\gG({1\mp2\pi \ii })
\gG(1\mp2\pi (e-1)\ii)
}
\\\notag&
=(1-e\qw)\Bigpar{\frac{\pi \cdot 2\pi}{\sinh(2\pi^2)} \cdot
\frac{\pi \cdot 2\pi (e-1)}{\sinh(2(e-1)\pi^2)} }^{1/2}
\\\notag&
\approx {4\pi^2(e-1)^{3/2}}{e^{-e\pi^2-1}}
\doteq 0.730\cdot 10^{-10}.
  \end{align}
Hence, $\qx_1(u)$ varies about its `mean' $1-e^{-1}$ with small amplitude
of the order $10^{-10}$.
The oscillations are somewhat larger for larger $j$, but still small, and
thus $\qx_j(u)$ is well approximated by its mean $(1-e^{-1})^j/j$.
Hence, \refT{TQ} implies that for large $n$, the distribution of $Q_n$ is
for practical purposes
well approximated by the logarithmic distribution
\begin{align}
  \P[Q_n=j]\approx \frac{(1-e^{-1})^j}{j},
\qquad j=1,2,\dots
\end{align}
In particular, the maximum is unique with probability close to $1-e^{-1}$.

As mentioned in the proof,  $q_j$  has appeared in connection with the multiplicity of the maximum in a sample from Geometric$(1-e^{-1})$.
For this case Brands, Steutel and Wilms \cite[Remark 2]{Brands} observe that the fluctuations of $q_1$ are on the scale $10^{-4}$. For
the stopped maximum occupancy count these are smaller, again due to the smoothing resulting from the randomisation.

\section{Numerics}
In Table~\ref{t1}, we compare, for the CCP case $m = \ell = 0$, the results
of simulations of $M_n$ with the expectation $E_n$ of the approximation
in~\eqref{ti3}. By~\eqref{bc} and~\eqref{Zni}, this approximation can be
written as $b_n + Z_{n,1}$; thus, $E_n= b_n + \E[Z_{n,1}]$. Furthermore,
by~\eqref{mean-Z} and~\eqref{wi1}, $E_n$ equals
\begin{align}
a_n + \gamma e + \frac{1}{2}
\end{align}
within the precision used here; see also~\eqref{mean-M}.
Due to computational complexity, we used $100{,}000$ Monte Carlo simulations for $n = 10$ and $n = 100$; $10{,}000$ simulations for $n = 10^3, 10^4, 10^5, 150{,}000$; $1{,}000$ simulations for $n = 200{,}000; 250{,}000; 500{,}000$; and $300$ simulations for $n = 10^6$. For each considered value of $n$, we report the mean of the simulated value $\widetilde{M_n}$ of $M_n$, and the standard deviation of this mean, in columns three and four, respectively.

\begin{table*}[h!]
\caption{}
\label{t1}
\begin{center}
\begin{tabular}{rccccc}
  \hline
  $n$ & $E_n$
& mean($\widetilde{M_n}$) & std(mean($\widetilde{M_n}$))\\
  \hline
  10 &                              \phantom{0}5.95083                  & \phantom{0}5.71614         &  0.00677 \\
  100&                                         11.86333                 &  11.66791                  &  0.00750\\
  1,000&                                       17.91967                 &  17.76950                  &  0.02410\\
  10,000&                                      24.03490                 &  23.90320                  &  0.02471 \\
  100,000&                                     30.18241                 &  30.10330                  &  0.02494\\
  150,000&                                     31.26727                 &  31.17440                  &  0.02457\\
  200,000&                                     32.03735                 &  32.09200                  &  0.08153  \\
  250,000&                                     32.63485                 &  32.45600                  &  0.07721\\
  500,000&                                     34.49189                 &  34.44200                  &  0.08071\\
  1,000,000&                                   36.35032                 &  36.12333                  &  0.15089\\
  \hline
\end{tabular}
\end{center}
\end{table*}

In Table~\ref{t2}, we compare, for various pairs $(\ell, m)$ and for $n = 10^2, 10^3, 10^4, 10^5$, the results of simulations of $M_n$ with the expected value $E_n$ from the approximation in~\eqref{ti3}.
In a similar manner to Table~\ref{t1}, and using~\eqref{mean-Z} and~\eqref{wi1}, the quantity $E_n$ equals
\begin{align}
  a_n + \gamma e - (e - 1)\sum_{j=1}^{\ell} \frac{1}{j} + \frac{1}{2}
\end{align}
within the level of precision used here; see also~\eqref{mean-M}.
The number of Monte Carlo simulations is abbreviated by MC and reported in the table.
For each considered triple $(n, \ell, m)$, we report the mean of the simulated value $\widetilde{M_n}$ of $M_n$ in the sixth column, and the standard deviation of this mean in the seventh column.

\begin{table*}[h!]
\caption{}
\label{t2}
\centering
\begin{tabular}{cccc|ccc}
\toprule
$n$ & MC & $\ell$ & $m$ & $E_n$ & mean($\widetilde{M_n}$) & std(mean($\widetilde{M_n}$))\\
\midrule
\multirow{6}{*}{100} & \multirow{6}{*}{$10^5$}     & 0 & 1 &  14.48746       &       14.83162        &    0.00818 \\
                     &                             & 0 & 2 &  15.92055       &       17.48246        &    0.00876  \\
                     &                             & 0 & 3 &  16.65695       &       19.86110        &    0.00922 \\
                     &                             & 5 & 0 &  7.93992        &       7.94814         &    0.00412\\
                     &                             & 10& 0 &  6.83053        &       6.77567         &    0.00358  \\
                     &                             & 25& 0 &  5.30644        &       5.03896         &    0.00289 \\
\midrule
\multirow{6}{*}{1,000} & \multirow{6}{*}{$10^4$}    & 0 & 1     & 21.24050       &   21.48190        &    0.02576   \\
                       &                            & 0 & 2     & 23.37030       &   24.58420        &    0.02748 \\
                       &                            & 0 & 3     & 24.80341       &   27.27850        &    0.02844 \\
                       &                            & 10 & 0    & 12.88688       &   12.88840        &    0.01240  \\
                       &                            & 50 & 0    & 10.18877       &   9.94790         &    0.01108  \\
                       &                            & 100 & 0   & 9.00629        &    8.54890        &   0.00926 \\
\midrule
\multirow{6}{*}{10,000} & \multirow{6}{*}{$10^3$}      & 0     & 1 & 27.85005 &     28.16300 & 0.08306 \\
                        &                              & 0     & 3 & 32.40160 &     34.61200 & 0.08886 \\
                        &                              & 0     & 5 & 34.88438 &     39.89800 & 0.09507 \\
                        &                              & 10    & 0 & 19.00211 &     19.00500 & 0.04169 \\
                        &                              & 100   & 0 & 15.12153 &     14.89000 & 0.03272 \\
                        &                              & 1,000 & 0 & 11.17276 &     10.08900 & 0.02649 \\
\midrule
\multirow{6}{*}{100,000} & \multirow{6}{*}{$10^2$}     & 0     & 1 & 34.38098  & 34.72000 & 0.23956 \\
                         &                             & 0     & 2 & 37.38853  & 38.31000 & 0.32495\\
                         &                             & 0     & 3 & 39.69937  & 41.60000 & 0.26967 \\
                         &                             & 0     & 5 & 42.94900  & 47.39000 & 0.25776 \\
                         &                             & 100   & 0 & 21.26903  & 21.23000 & 0.13015 \\
                         &                             & 1,000 & 0 & 17.32026  & 16.73000 & 0.09832 \\
                         &                             & 10,000& 0 & 13.36454  & 11.47000 & 0.08343 \\
\bottomrule
\end{tabular}
\end{table*}

\begin{acks}[Acknowledgments]
The first author is indebted to Alexander Marynych for useful discussions and pointers to the literature.
Special thanks go to Petr Akhmet'ev for locating and sending us a copy of \cite{Ivch1}.
The third author would like to thank Boris Alemi for running the MATLAB program on the department’s computer.
\end{acks}

\begin{funding}
The second author is funded by
the Knut and Alice Wallenberg Foundation
and
the Swedish Research Council (VR). The third author is supported in part by BSF grant 2020063.
\end{funding}

\end{document}